\documentclass[12pt]{article}
\usepackage{amssymb}
\usepackage{amsmath}
\usepackage{amsthm}
\usepackage{color}
\usepackage{comment}
\usepackage{enumitem}
\usepackage{geometry}		
\usepackage{graphicx}
\usepackage[sort&compress]{natbib}     

\makeatletter
	\@addtoreset{equation}{section}
\makeatother

\let\originalleft\left
\let\originalright\right
\renewcommand{\left}{\mathopen{}\mathclose\bgroup\originalleft}
\renewcommand{\right}{\aftergroup\egroup\originalright}

\begin{document}

\newcommand\cP{\mathcal{P}}
\newcommand\cS{\mathcal{S}}
\newcommand{\rD}{{\rm D}}
\newcommand{\re}{{\rm e}}
\newcommand{\ri}{{\rm i}}
\newcommand{\ee}{\varepsilon}

\newtheorem{theorem}{Theorem}[section]
\newtheorem{corollary}[theorem]{Corollary}
\newtheorem{lemma}[theorem]{Lemma}
\newtheorem{proposition}[theorem]{Proposition}
\newtheorem{conjecture}[theorem]{Conjecture}
\newtheorem{algorithm}[theorem]{Algorithm}

\theoremstyle{definition}
\newtheorem{definition}{Definition}[section]
\newtheorem{example}[definition]{Example}

\theoremstyle{remark}
\newtheorem{remark}{Remark}[section]

\title{
One-dimensional first return maps for the two-dimensional border-collision normal form with a zero determinant.
}
\author{
D.J.W.~Simpson\\\\
School of Mathematical and Computational Sciences\\
Massey University\\
Palmerston North, 4410\\
New Zealand
}

\maketitle



\begin{abstract}

The two-dimensional border-collision normal form is a four-parameter family of continuous, piecewise-linear maps. When this form has a zero determinant, all of its nonlinear dynamics are captured a one-dimensional first return map. The first return map is discontinuous and piecewise-linear, where each piece of the map corresponds to a constant return time. We show that when the normal form has a repelling focus fixed point, the configuration of the first return map is dictated by a rational rotation number whereby the set of return times and ordering of the pieces of the map are given by the denominators of the left and right sequences of Farey parents of this number. The result is proved by characterising polygons formed from preimages of the switching manifold, and employing an inductive argument on the Farey web. Several surfaces in parameter space where the configuration changes are bifurcations from chaotic to quasiperiodic or mode-locked dynamics.

\end{abstract}

\section{Introduction}
\label{sec:intro}

Maps ${\bf x} \mapsto f({\bf x})$ describe how the state {\bf x}
of a dynamical system evolves in discrete time steps.
The long-term behaviour of a map is governed by its attractors,
and, as parameters are varied, attractors undergo qualitative changes
at critical parameter values, termed bifurcations.

Piecewise-smooth maps exhibit {\em border-collision bifurcations} (BCBs)
that occur when a fixed point hits a switching manifold, where the map is nonsmooth \cite{DiBu08}.
The change in dynamics brought about by a BCB is almost limitless:
a stable fixed point can change to a chaotic attractor \cite{NuYo92},
of any dimension \cite{Gl15b}, or to any number of coexisting chaotic attractors \cite{Si24c}.

If a map is continuous and piecewise-differentiable at a BCB,
then the local dynamics are usually captured by a piecewise-linear approximation to the map \cite{Si16}.
This approximation is formed by truncating each piece of the map to first order.
If the approximation satisfies a genericity condition (observability),
then it can be converted via a change of coordinates to the border-collision normal form \cite{Di03}.
In two dimensions, this form is
\begin{equation}
\begin{bmatrix} x \\ y \end{bmatrix} \mapsto
\begin{cases}
\begin{bmatrix} \tau_L x + y + \mu \\ -\delta_L x \end{bmatrix}, & x \le 0, \\
\begin{bmatrix} \tau_R x + y + \mu \\ -\delta_R x \end{bmatrix}, & x \ge 0,
\end{cases}
\label{eq:bcnf}
\end{equation}
which has four parameters, $\tau_L, \delta_L, \tau_R, \delta_R \in \mathbb{R}$,
in addition to the border-collision parameter $\mu \in \mathbb{R}$.
The cases $\mu < 0$ and $\mu > 0$ correspond to different sides of the BCB.
All attractors and bounded invariant sets of \eqref{eq:bcnf} scale linearly with $|\mu|$,
thus to understand the dynamics associated with the BCB,
it suffices to consider $\mu = -1$ and $\mu = 1$.

This paper concerns the case that one of the determinants $\delta_L$ or $\delta_R$ is zero,
and without loss of generality we take $\delta_L = 0$.
From an applied perspective, this situation is generic
because BCBs corresponding to grazing-sliding bifurcations of Filippov systems
always involve a zero determinant \cite{DiKo03}.
This occurs because any piece of a Poincar\'e map corresponding to trajectories with sliding segments,
return to a codimension-one subset of the Poincar\'e section.
Grazing-sliding bifurcations are ubiquituous in
low-dimensional models of mechanical systems with stick-slip friction
\cite{CsSt06,GuHo10,KoPi08,LuGe06,YoSu00},
and ecological systems with threshold control \cite{HaAr21,ZhTa22}.

For \eqref{eq:bcnf} with a zero determinant,
Parui and Banerjee \cite{PaBa02} systematically catalogued the BCBs.
They achieved this by performing extensive numerical explorations to
determine what attractors are possible on each side of the BCB
for every possible combination of stability types for the two fixed points.
In particular, this work highlighted the complexity of the possible
dynamics despite the presence of a zero determinant.

To understand how complex dynamics arises,
our previous work \cite{Si25f} catalogued the bifurcation structures occurring for fixed $\mu \ne 0$.
It was found that with $\delta_L = 0$ and $\mu = -1$,
\eqref{eq:bcnf} behaves similarly to the one-dimensional setting 
whereby chaos and period-incrementing structures dominate \cite{ItTa79,NuYo95,SuAv16},
but with two dimensions multiple attractors are possible. 
With instead $\delta_L = 0$ and $\mu = 1$, the dynamics of \eqref{eq:bcnf} can be extraordinarily rich.
In this case, \eqref{eq:bcnf} reduces to
\begin{equation}
\begin{bmatrix} x \\ y \end{bmatrix} \mapsto f(x,y) =
\begin{cases}
\begin{bmatrix} \tau_L x + y + 1 \\ 0 \end{bmatrix}, & x \le 0, \\
\begin{bmatrix} \tau_R x + y + 1 \\ -\delta_R x \end{bmatrix}, & x \ge 0,
\end{cases}
\label{eq:f}
\end{equation}
and Fig.~\ref{fig:bifSet} shows a sample two-dimensional slice of parameter space.
The figure shows robust chaos (orange),
coexisting attractors (speckled),
and the sausage-string structure of periodicity regions,
studied by Szalai and Osinga \cite{SzOs09}.
A different structure for the periodicity regions
occurs in the left area of the figure, and this structure remains to be properly understood.

\begin{figure}[t!]
\begin{center}
\includegraphics[width=15cm]{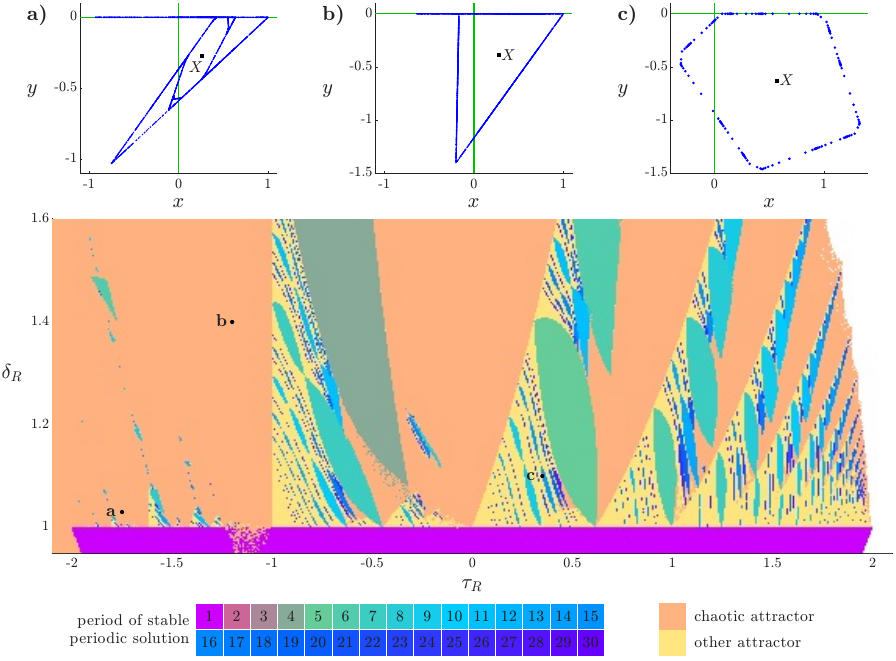}
\caption{
A two-parameter bifurcation diagram and sample phase portraits of
\eqref{eq:f} with $\tau_L = -1.2$.
Each point in a $2000 \times 1000$ grid of $(\tau_R,\delta_R)$-values
is coloured according to the long-term behaviour of the forward orbit of a random initial point.
If the orbit appeared to converge to a periodic solution with period less than or equal to $30$,
the point is coloured by the period as indicated in the colour bar.
Otherwise the point is coloured white if the orbit appeared to diverge,
orange if a numerically computed maximal Lyapunov exponent is greater than $0.001$,
and yellow otherwise.
The phase portraits use
$(\tau_R,\delta_R) = (-1.75,1.03)$ in {\bf a},
$(\tau_R,\delta_R) = (-1.2,1.4)$ in {\bf b}, and
$(\tau_R,\delta_R) = (0.35,1.1)$ in {\bf c}.
In these plots the switching manifold (the $y$-axis, $\Sigma_0$) and its image (the $x$-axis, $\Sigma_1$)
are indicated with green lines.
The black square is the repelling fixed point $X$,
while the blue dots indicate the attractor.
The attractor appears to be chaotic at parameter points {\bf a} and {\bf b},
and periodic with period $149$ at parameter point {\bf c}.
\label{fig:bifSet}
} 
\end{center}
\end{figure}

The left piece of \eqref{eq:f} maps all points to the $x$-axis,
\begin{equation}
\Sigma_1 = \left\{ (x,y) \in \mathbb{R}^2 \,\middle|\, y = 0 \right\}.
\label{eq:Sigma1}
\end{equation}
Any orbit of \eqref{eq:f} that does not become trapped in the right-half plane visits $\Sigma_1$ repeatedly,
thus all invariant sets with nonlinear dynamics
are captured by the induced map defined by first return to $\Sigma_1$.
This map is one-dimensional,
so we anticipate that one-dimensional techniques \cite{DeVa93,Ru17}
can be used to obtain a rigorous and detailed explanation
for the attractors and bifurcation structures of \eqref{eq:f}.
Indeed Kowalczyk \cite{Ko05} has already used the return map
to demonstrate the existence of robustly chaotic attractors,
but to apply it further we first need to understand the configuration of the map,
as this differs substantially across parameter space.

\begin{figure}[t!]
\begin{center}
\includegraphics[width=15cm]{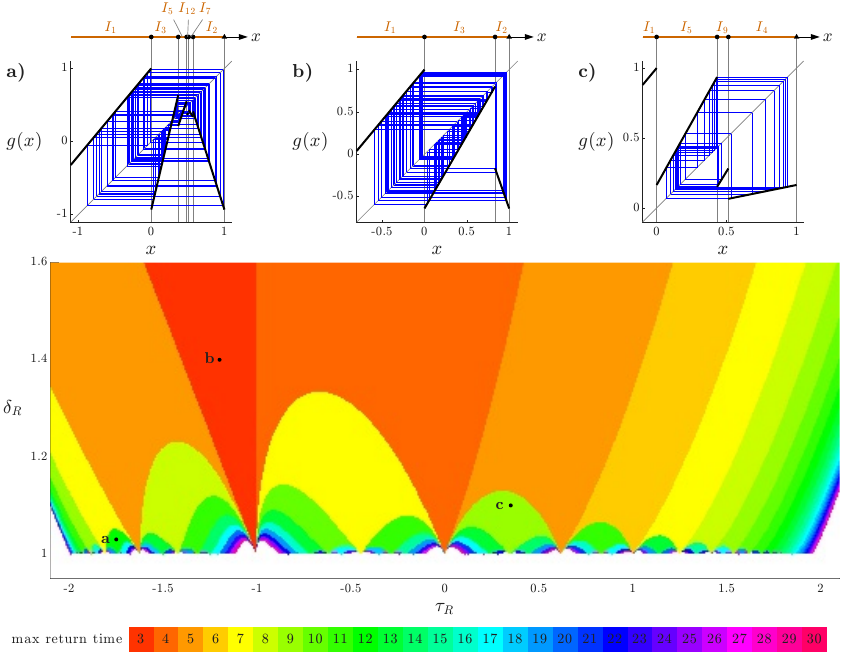}
\caption{
The lower plot indicates the maximum number of iterations (return time)
required for orbits of \eqref{eq:f} with
$\delta_R > 1$ and $\delta_R > \frac{\tau_R^2}{4}$ to return to the $x$-axis, $\Sigma_1$.
The upper plots show the first return map at the parameter points of Fig.~\ref{fig:bifSet},
also indicating the intervals $I_n$ on which this map is affine.
In each plot, the blue cobweb corresponds to the attractor of the return map.
\label{fig:numPieces}
} 
\end{center}
\end{figure}

The upper plots in Fig.~\ref{fig:numPieces}
are cobweb diagrams of the first return map
at the parameter points {\bf a}, {\bf b}, and {\bf c} of Fig.~\ref{fig:bifSet}.
The return map is discontinuous
and affine on a finite set of intervals $I_n$,
corresponding to fixed return times $n \ge 1$.
The maximum return time is $12$, $3$, and $9$,
at parameter points {\bf a}, {\bf b}, and {\bf c}, respectively,
and Fig.~\ref{fig:numPieces} also shows how the maximum return time
varies with $\tau_R$ and $\delta_R$.

The purpose of this paper is to
obtain a concise characterisation for the configuration of the first return map.
Our results are stated in \S\ref{sec:mainResults} in the form of two theorems.
Theorem \ref{th:existence} shows there exists an irreducible fraction $\frac{m}{p}$,
such that the maximum return time is $p$
and that the other return times and ordering of the intervals $I_n$ are given
by the denominators of the left and right sequences of Farey parents of $\frac{m}{p}$.
Theorem \ref{th:construct} shows how the intervals $I_n$ can be obtained constructively.
The theorems apply when right piece of \eqref{eq:f} has a repelling focus fixed point,
which occurs when $\delta_R > 1$ and $\delta_R > \frac{\tau_R^2}{4}$,
and covers the most dynamically interesting parts of parameter space.

In \S\ref{sec:Farey} we review the Farey organisation of rational numbers.
We define Farey neighbours, Farey parents, and the Farey web,
and relate irreducible fractions to permutations.

Theorems \ref{th:existence} and \ref{th:construct} are proved in \S\ref{sec:arguments}.
This is achieved in several steps.
We first study polygons $E_r$ formed from preimages of the switching manifold of \eqref{eq:f}.
Each $E_r$ consists of all points in the right half-plane
that require exactly $r$ iterations under \eqref{eq:f} to enter the closed left half-plane.
Then each $I_n$ is formed by intersecting $E_{n-1}$ with $\Sigma_1$ at points with $x \le 1$.
Our constructions determine the values of $n$ for which $I_n$ is non-empty,
and the way in which the non-empty intervals are ordered.
At its core, this requires a proof by induction, taking steps down the Farey web.

Additional results are presented in \S\ref{sec:more}.
We use the geometric constructions to verify a statement of Szalai and Osinga \cite{SzOs08}
related to the polygons $F_r = \mathbb{R}^2 \setminus (E_0 \cup E_1 \cup \cdots \cup E_{r-1})$.
We then derive an explicit formula for the boundaries
in Fig.~\ref{fig:numPieces} where there is change to the maximum return time
and configuration of $g$.
We also describe how these boundaries can coincide with bifurcations
where the attractor changes qualitatively.
Section \ref{sec:conc} provides concluding remarks and an outlook for future studies.
Throughout the paper we use parameter points {\bf a}, {\bf b}, and {\bf c}
of Fig.~\ref{fig:bifSet} to illustrate the results.

\section{Main results}
\label{sec:mainResults}

This section presents the main results.
We write
\begin{align}
\Omega_L &= \left\{ (x,y) \in \mathbb{R}^2 \,\middle|\, x < 0 \right\}, \\
\Omega_R &= \left\{ (x,y) \in \mathbb{R}^2 \,\middle|\, x > 0 \right\},
\label{eq:OmegaLR}
\end{align}
for the open left and right half-planes,
and $\overline{\Omega}_L$ and $\overline{\Omega}_R$
for the closed left and right half-planes.
We also write $P = (P_1,P_2)$ to indicate the components of a point $P \in \mathbb{R}^2$.

\subsection{First return}
\label{sub:return}

For any $\tau_L, \tau_R, \delta_R \in \mathbb{R}$,
the map $f$, given by \eqref{eq:f}, maps every point in $\overline{\Omega}_L$ to
the $x$-axis $\Sigma_1$.
The first return time $N$ and map $g$ are defined as follows.

\begin{definition}
Given $x \in \mathbb{R}$, the {\em return time} $N(x)$
is the smallest $n \ge 1$ for which $f^n(x,0)_2 = 0$, if such $n$ exists.
If $n = N(x)$ exists, the {\em first return map} has value
\begin{equation}
g(x) = f^n(x,0)_1 \,,
\label{eq:g1}
\end{equation}
otherwise $g(x)$ is undefined.
\label{df:firstReturnMap}
\end{definition}

If $\delta_R \ne 0$ and $g(x)$ is undefined, then the forward orbit of $(x,0)$ remains in $\Omega_R$.
Since $f$ is affine on $\Omega_R$, the behaviour of such orbits
is easily determined from the values of $\tau_R$ and $\delta_R$.
For example, if $|\tau_R| - 1 < \delta_R < 1$,
the orbits converge to a fixed point in $\Omega_R$.

The following result shows that, throughout the parameter regimes we wish to consider,
$g$ only takes values in $(-\infty,1]$.
This allows us to restrict our analysis of $g$ to $(-\infty,1]$, providing simplification.
This result is proved at the end of this section.

\begin{lemma}
Let $\tau_L \ge 0$, $\tau_R \in \mathbb{R}$, and $\delta_R > 0$.
Then $g(x) \le 1$ for all $x \in \mathbb{R}$ for which $g(x)$ is defined.
\label{le:xPrimeLe1}
\end{lemma}

For all $n \ge 1$, let
\begin{equation}
I_n = \left\{ x \in (-\infty,1] \,\big|\, N(x) = n \right\},
\label{eq:In}
\end{equation}
be the set of all $x$-values that yield a return time of $n$.
The next result, proved at the end of this section,
shows that each non-empty $I_n$ is either an interval or the singleton $\{ 1 \}$.

\begin{lemma}
Let $\tau_L, \tau_R, \delta_R \in \mathbb{R}$ and $n \ge 1$.
If $I_n \ne \varnothing$, and $I_n \ne \{ 1 \}$, then $I_n$ is an interval. 
\label{le:InInterval}
\end{lemma}

\begin{figure}[b!]
\begin{center}
\includegraphics[width=15cm]{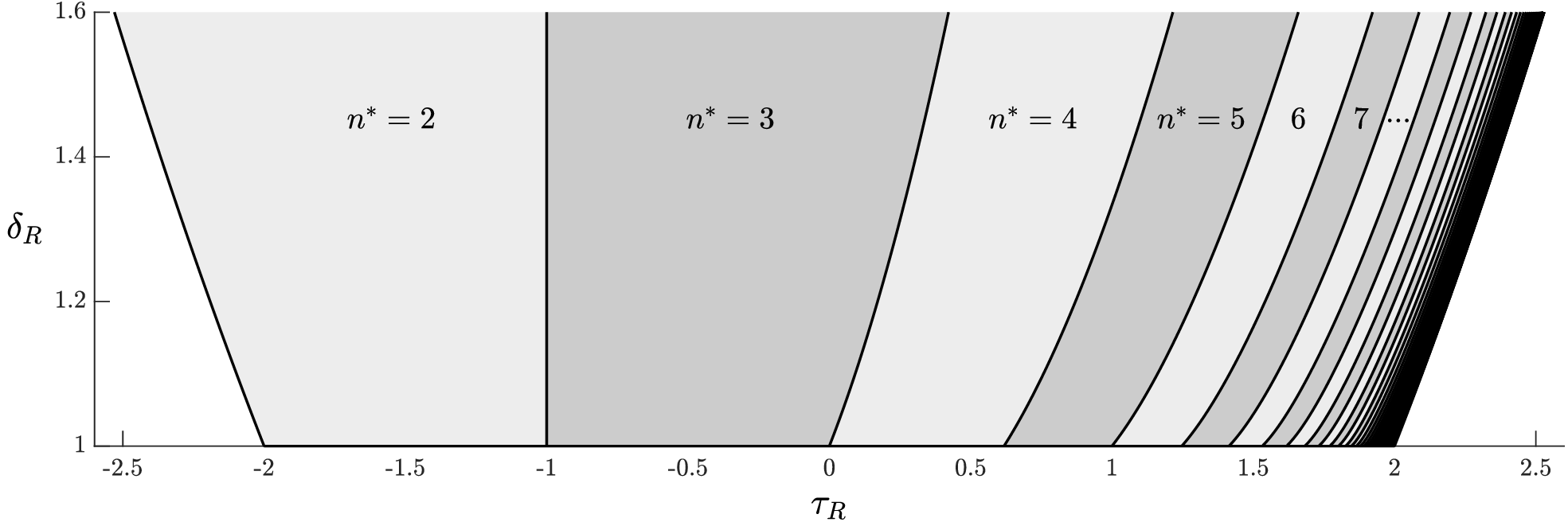}
\caption{
Regions where the return time $n^*$ of $x=1$ is constant. 
\label{fig:nStarDivision}
} 
\end{center}
\end{figure}

The return time for $x=1$,
\begin{equation}
n^* = N(1),
\label{eq:nStar}
\end{equation}
has particular significance, as will become clear below.
Fig.~\ref{fig:nStarDivision} shows how the value of $n^*$ varies with $\tau_R$ and $\delta_R$.
For every $k \ge 2$, the boundary of neighbouring regions $n^* = k$ and $n^* = k+1$
is where the point $(x,y) = (1,0)$
maps under $k-1$ iterations of $f$ to the switching manifold of $f$ (the $y$-axis).

\subsection{The existence of a rational rotation number}
\label{sub:existence}

\begin{figure}[b!]
\begin{center}
\includegraphics[width=15cm]{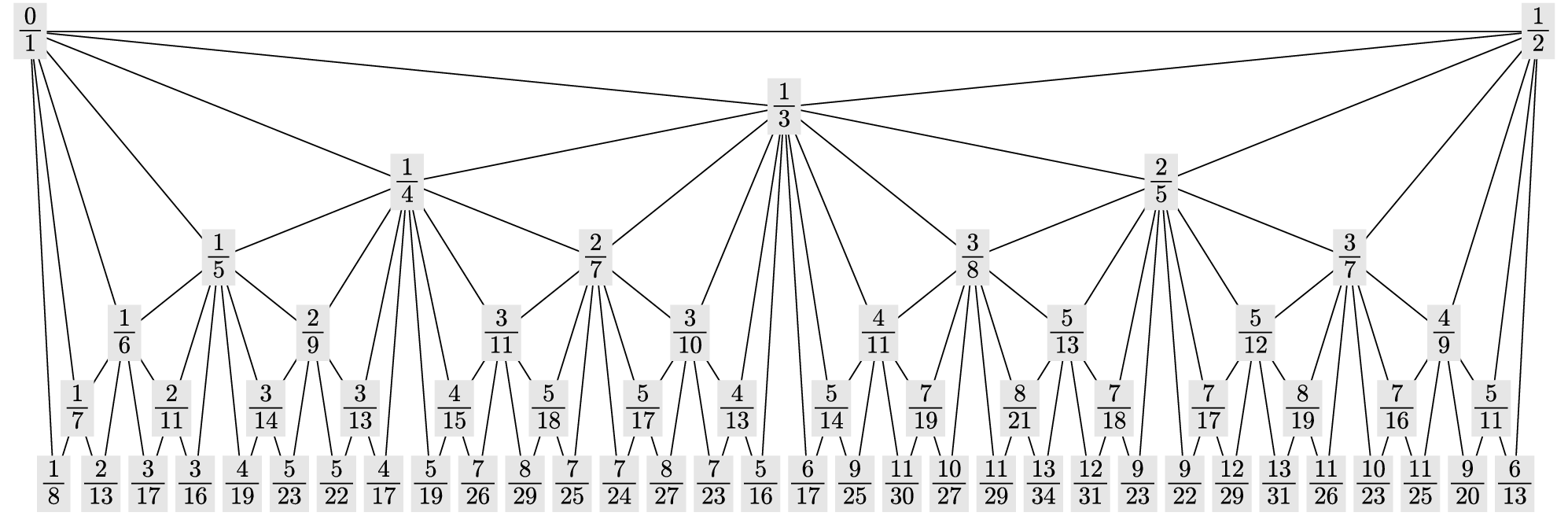}
\caption{
The Farey web for rational numbers in $\left[ 0, \frac{1}{2} \right]$ up to level six.
Level zero consists of $\frac{0}{1}$ and $\frac{1}{2}$.
For each $\ell \ge 1$, level $\ell$ is constructed by evaluating the Farey sum
on all Farey neighbours that have a value in level $\ell-1$.
\label{fig:sternBrocot}
} 
\end{center}
\end{figure}

The {\em Farey web} is indicated in Fig.~\ref{fig:sternBrocot}
for irreducible fractions $\frac{m}{p} \in \left[ 0, \frac{1}{2} \right]$.
In short, the Farey web can be generated by iteratively applying the Farey sum
$\frac{m^- + m^+}{p^- + p^+}$ to {\em neighbouring} values $\frac{m^-}{p^-}$ and $\frac{m^+}{p^+}$. 
A full definition is provided in \S\ref{sub:web}.

Given irreducible $\frac{m}{p} \in \left( 0, \frac{1}{n^*} \right)$,
consider the path in the Farey web from $\frac{m}{p}$ upwards to $\frac{0}{1}$,
and the path from $\frac{m}{p}$ upwards to $\frac{1}{n^*}$.
Order the numbers in these paths numerically,
and let $D[m,p;n^*]$ denote their denominators. 
For example, with $n^* = 2$ and $\frac{m}{p} = \frac{5}{12}$,
the paths yield $\frac{0}{1} < \frac{1}{3} < \frac{2}{5} < \frac{5}{12} < \frac{3}{7} < \frac{1}{2}$,
so $D[5,12;2] = (1,3,5,12,7,2)$.
The list $D[m,p;n^*]$ is defined more precisely in \S\ref{sub:web}.

\begin{theorem}
Let $\tau_L, \tau_R, \delta_R \in \mathbb{R}$, with $\delta_R > 1$ and $\delta_R > \frac{\tau_R^2}{4}$.
There exists a unique, $\tau_L$-independent, irreducible fraction $\frac{m}{p} \in \left[ \frac{1}{n^*+1}, \frac{1}{n^*} \right)$
such that for all $n \ge 1$ the interval $I_n$ is non-empty if and only if $n \in D[m,p;n^*]$.
The non-empty $I_n$ cover $(-\infty,1]$, and are ordered by $D[m,p;n^*]$.
\label{th:existence}
\end{theorem}

\subsection{Remarks on Theorem \ref{th:existence}}
\label{sub:remarks}

Theorem \ref{th:existence} shows how the configuration of the first return map $g$
is characterised by an irreducible fraction $\frac{m}{p}$.
This value is independent of $\tau_L$
because the return time to $\Sigma_1$
is one greater than the time taken to enter $\overline{\Omega}_L$.
Before entering $\overline{\Omega}_L$, orbits evolve purely under the right piece of $f$,
hence the intervals $I_n$ are independent of the parameter $\tau_L$ of the left piece of $f$.
The value $\tau_L$ affects only the slopes and vertical positioning of the pieces of $g$.

In Theorem \ref{th:existence}, the non-empty $I_n$ cover $(-\infty,1]$,
so $g(x)$ is defined for all $x \in (-\infty,1]$.
This occurs because with $\delta_R > 1$ and $\delta_R > \frac{\tau_R^2}{4}$
the map has repelling focus fixed point belonging $\Omega_R$ but not $\Sigma_1$,
thus forward orbits cannot become trapped in $\Omega_R$.

Following Szalai and Osinga \cite{SzOs08},
we refer to $\frac{m}{p}$ as a {\em rotation number}.
This is because the set $F_p$ of all points in $\Omega_R$
that require at least $p$ iterations to enter $\overline{\Omega}_L$
is a convex $p$-sided polygon whose sides
are naturally ordered via the fraction $\frac{m}{p}$, see \S\ref{sub:minimalPolygons}.
This polygon, roughly speaking, sits inside the attractor of $f$.

The largest value in $D[m,p;n^*]$ is $p$, hence $p$ is the maximum return time.
Fig.~\ref{fig:numPiecesWithRotNums} repeats Fig.~\ref{fig:numPieces},
but now labels the regions by their rotation number $\frac{m}{p}$.

The fixed point of $f$ belongs to $F_p$, and has rotation number
$\rho_{\rm fp} = \frac{\phi}{2 \pi}$, where
\begin{equation}
\phi = \cos^{-1} \left( \frac{\tau_R}{2 \sqrt{\delta_R}} \right) \in (0,\pi).
\label{eq:phi}
\end{equation}
The value $\rho_{\rm fp}$ is irrational for almost all values of $\tau_R$ and $\delta_R$.
Thus $\frac{m}{p}$ can be interpreted as a rationalisation of $\rho_{\rm fp}$
brought about by the presence of the switching manifold.

\begin{figure}[b!]
\begin{center}
\includegraphics[width=15cm]{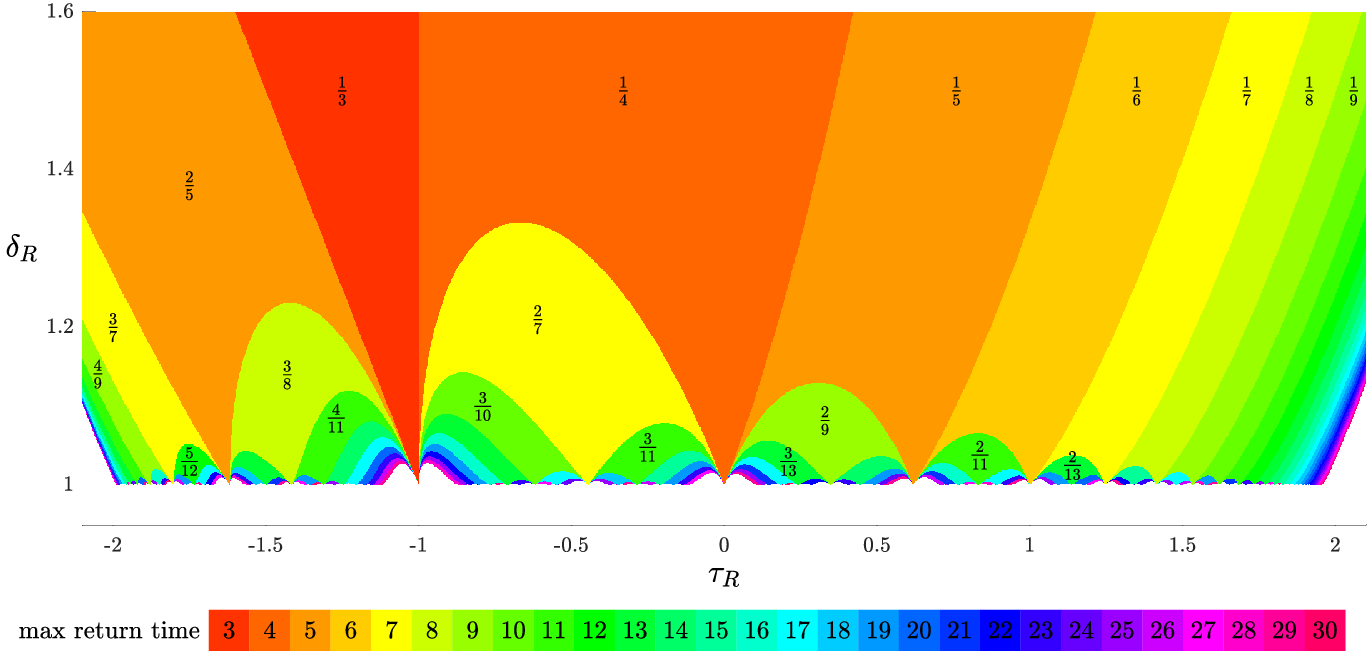}
\caption{
A repeat of Fig.~\ref{fig:numPieces}
but with regions labelled by the rotation number $\frac{m}{p}$ of Theorem \ref{th:existence}.
\label{fig:numPiecesWithRotNums}
} 
\end{center}
\end{figure}

\subsection{Constructing the rotation number and the non-empty intervals}
\label{sub:constr}

Let
\begin{align}
f_L(x,y;\tau_L) &= \begin{bmatrix} \tau_L x + y + 1 \\ 0 \end{bmatrix}, &
f_R(x,y;\tau_R,\delta_R) &= \begin{bmatrix} \tau_R x + y + 1 \\ -\delta_R x \end{bmatrix},
\label{eq:fLfR}
\end{align}
denote the left and right pieces of \eqref{eq:f}.
With $\delta_R \ne 0$, the forward orbit of $(x,0)$
cannot return to $\Sigma_1$ until entering $\overline{\Omega}_L$,
upon which it reaches $\Sigma_1$ on the next iteration.
In this case $g$ can be written as follows.

\begin{lemma}
Let $\tau_L, \tau_R, \delta_R \in \mathbb{R}$ with $\delta_R \ne 0$, and let $x \in \mathbb{R}$.
If $g(x)$ is defined, then
\begin{equation}
g(x) = f_L \left( f_R^r(x,0) \right)_1 \,,
\label{eq:g2}
\end{equation}
where $r = N(x) - 1$.
\label{le:rEqualsNminus1}
\end{lemma}

In \eqref{eq:g2}, $r$ is the smallest non-negative integer
for which $f_R^r(x,0) \in \overline{\Omega}_L$.
Since $f_R$ is affine, it is simple exercise
to obtain the following explicit expression for
the first component of $f_R^r(x,0)$.

\begin{lemma}
Let $\tau_R, \delta_R \in \mathbb{R}$ with $\delta_R > \frac{\tau_R^2}{4}$.
Given $r \ge 0$, the first component of $f_R^r(x,0)$ is 
\begin{equation}
f_R^r(x,0)_1 = \alpha_r x + \beta_r \,,
\label{eq:fRr1}
\end{equation}
where
\begin{equation}
\begin{split}
\alpha_r &= \frac{\delta_R^{\frac{r}{2}} \sin((r+1) \phi)}{\sin(\phi)}, \\
\beta_r &= \frac{\sin(\phi) - \delta_R^{\frac{r}{2}} \sin((r+1) \phi)
+ \delta_R^{\frac{r+1}{2}} \sin(r \phi)}{(\delta_R - \tau_R + 1) \sin(\phi)},
\end{split}
\label{eq:alphabeta}
\end{equation}
and $\phi$ is given by \eqref{eq:phi}.
\label{le:alphabeta}
\end{lemma}

If $\alpha_r \ne 0$, then the unique value of $x \in \mathbb{R}$
for which $f_R^r(x,0)_1 = 0$ is $-\frac{\beta_r}{\alpha_r}$.
The values $-\frac{\beta_r}{\alpha_r}$ form the endpoints of the intervals $I_n$, and we write
\begin{equation}
z_n = -\frac{\beta_{n-1}}{\alpha_{n-1}},
\label{eq:zn}
\end{equation}
for any $n \ge 1$ for which $\alpha_{n-1} \ne 0$.

The following result provides a practical way of computing the intervals $I_n$.
Recall, $n^*$ is the value \eqref{eq:nStar}.
Note, $\alpha_{n^* - 1} \ne 0$ and $z_{n^*} \in (0,1]$,
by part 1 of Theorem \ref{th:construct}.

\begin{theorem}
Let $\tau_L, \tau_R, \delta_R \in \mathbb{R}$, with $\delta_R > 1$ and $\delta_R > \frac{\tau_R^2}{4}$.
The following algorithm generates the rotation number $\frac{m}{p}$ of Theorem \ref{th:existence}
and all non-empty intervals $I_n$ of $g$.
\begin{enumerate}[label=\arabic*),ref=\arabic*,itemsep=0mm]
\item
Let $I_1 = (-\infty,0]$ and $(m^-,p^-,z^-) = (0,1,0)$.
Let $I_{n^*} = \left[ z_{n^*}, 1 \right]$,
where $z_{n^*} \in (0,1]$ is well-defined,
and let $(m^+,p^+,z^+) = \left( 1,n^*,z_{n^*} \right)$.
\item
Let $\frac{m}{p} = \frac{m^- + m^+}{p^- + p^+}$.
Evaluate $\alpha_{p-1}$, and if $\alpha_{p-1} \ne 0$ evaluate $\beta_{p-1}$ and $z_p$.
\item
If $\alpha_{p-1} = 0$ or $z_p \notin \left( z^-, z^+ \right)$, let $I_p = \left( z^-, z^+ \right)$ and {\sc stop}.
Otherwise, if $\alpha_{p-1} < 0$, let $I_p = \left[ z_p, z^+ \right)$, update $(m^+,p^+,z^+) = (m,p,z_p)$, and return to Step 2,
while if $\alpha_{p-1} > 0$, let $I_p = \left( z^-, z_p \right]$, update $(m^-,p^-,z^-) = (m,p,z_p)$, and return to Step 2.
\end{enumerate}
\label{th:construct}
\end{theorem}

The algorithm terminates when {\sc stop} is reached,
and this always occurs after finitely many computations.

\subsection{Example}
\label{sub:example}

Here we execute the algorithm in Theorem \ref{th:construct}
with $\tau_R = -1.75$ and $\delta_R = 1.03$, corresponding to
point {\bf a} of Fig.~\ref{fig:bifSet}.
First observe $n^* = 2$, see Fig.~\ref{fig:nStarDivision},
so $z_{n^*} = -\frac{\beta_1}{\alpha_1} = -\frac{1}{\tau_R} = 0.5714$;
here, and in the calculations below, numbers are rounded to four decimal places.

\begin{enumerate}[label=\roman*),ref=\roman*,itemsep=0mm]
\item
\begin{itemize}[itemsep=0mm]
\item
In Step 1 we let $I_1 = (-\infty,0]$ and $(m^-,p^-,z^-) = (0,1,0)$.
Also $I_2 = [0.5714,1]$ and $(m^+,p^+,z^+) = (1,2,0.5714)$.
\item
In Step 2 we have $\frac{m^-}{p^-} = \frac{0}{1}$ and $\frac{m^+}{p^+} = \frac{1}{2}$,
so the Farey sum gives $\frac{m}{p} = \frac{1}{3}$.
The formulas \eqref{eq:alphabeta} give $\alpha_2 = 2.0325$ and $\beta_2 = -0.75$,
then \eqref{eq:zn} gives $z_3 = 0.3690$.
\item
In Step 3, $z_3 \in \left( z^-, z^+ \right) = (0,0.5714)$ and $\alpha_2 > 0$.
Thus $I_3 = (0,0.3690)$, and we perform the update $(m^-,p^-,z^-) = (1,3,0.3690)$.
\end{itemize}
\item
\begin{itemize}
\item
Returning to Step 2, now $\frac{m^-}{p^-} = \frac{1}{3}$ and $\frac{m^+}{p^+} = \frac{1}{2}$,
so the Farey sum gives $\frac{m}{p} = \frac{2}{5}$.
The formulas \eqref{eq:alphabeta} give $\alpha_4 = 0.9767$ and $\beta_4 = -0.4719$,
then \eqref{eq:zn} gives $z_5 = 0.4831$.
\item
In Step 3, $z_5 \in \left( z^-, z^+ \right) = (0.3690,0.5714)$ and $\alpha_4 > 0$.
Thus $I_5 = (0.3690,0.4831)$, and we perform the update $(m^-,p^-,z^-) = (2,5,0.4831)$.
\end{itemize}
\item
\begin{itemize}
\item
Returning to Step 2, now $\frac{m^-}{p^-} = \frac{2}{5}$ and $\frac{m^+}{p^+} = \frac{1}{2}$,
so the Farey sum gives $\frac{m}{p} = \frac{3}{7}$.
The formulas \eqref{eq:alphabeta} give $\alpha_6 = -1.1772$ and $\beta_6 = 0.6026$,
then \eqref{eq:zn} gives $z_7 = 0.5119$.
\item
In Step 3, $z_7 \in \left( z^-, z^+ \right) = (0.4831,0.5714)$ and $\alpha_6 < 0$.
Thus $I_7 = (0.5119,0.5714)$, and we perform the update $(m^+,p^+,z^+) = (3,7,0.5119)$.
\end{itemize}
\item
\begin{itemize}
\item
Returning to Step 2, now $\frac{m^-}{p^-} = \frac{2}{5}$ and $\frac{m^+}{p^+} = \frac{3}{7}$,
so the Farey sum gives $\frac{m}{p} = \frac{5}{12}$.
The formulas \eqref{eq:alphabeta} give $\alpha_{11} = -0.2135$ and $\beta_{11} = 0.0559$,
then \eqref{eq:zn} gives $z_{12} = 0.2619$.
\item
In Step 3, $z_{12} \notin \left( z^-, z^+ \right) = (0.4831,0.5119)$,
thus $I_{12} = (0.4831,0.5119)$ and the algorithm terminates.
\end{itemize}
\end{enumerate}

\begin{figure}[t!]
\begin{center}
\includegraphics[width=15cm]{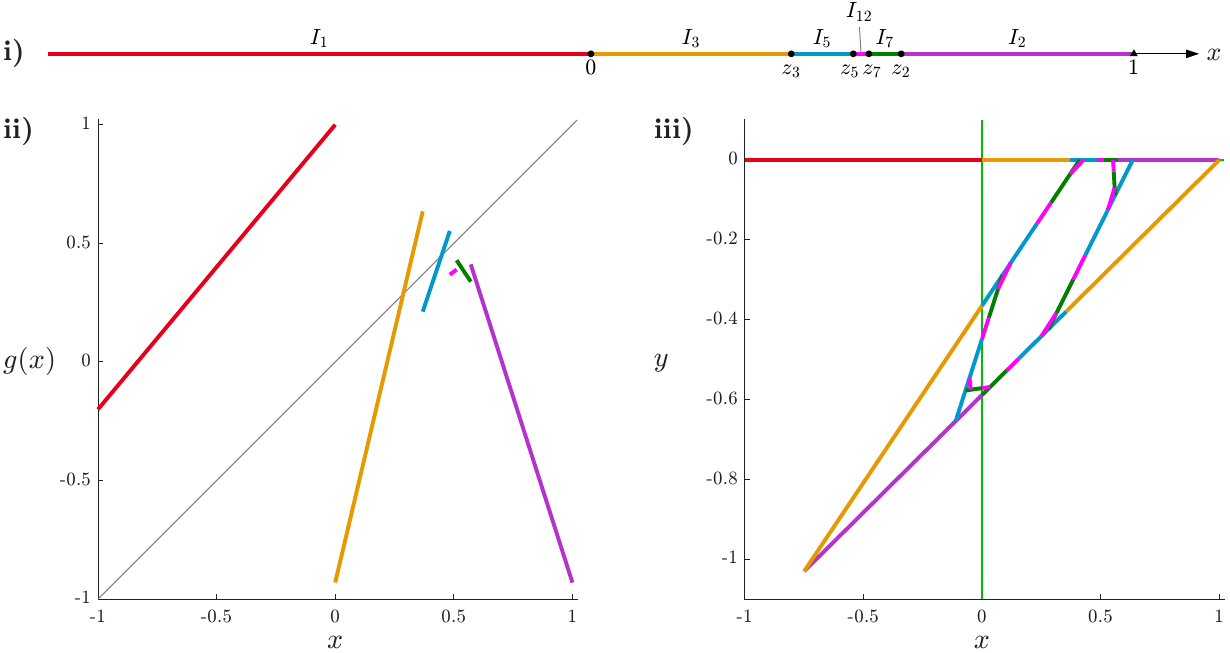}
\caption{
The return map $g$ and related constructions using
$(\tau_L,\tau_R,\delta_R) = (1.2,-1.75,1.03)$,
corresponding to point {\bf a} of Fig.~\ref{fig:bifSet}:
(i) shows the non-empty intervals $I_n$;
(ii) shows the return map $g$;
(iii) shows the images of $I_n \times \{ 0 \}$ under $f$.
\label{fig:invariantMultigon6}
} 
\end{center}
\end{figure}

The algorithm terminated with
rotation number $\frac{m}{p} = \frac{5}{12}$.
Fig.~\ref{fig:invariantMultigon6}(i) shows the computed intervals $I_n$.
In accordance with Theorem \ref{th:existence}
the intervals are ordered by $D[5,12;2] = (1,3,5,12,7,2)$
and cover the interval $(-\infty,1]$.

Fig.~\ref{fig:invariantMultigon6}(ii) shows the return map $g$ when $\tau_L = 1.2$.
With different values of $\tau_L$,
the pieces of $g$ have different slopes and vertical positionings.
Fig.~\ref{fig:invariantMultigon6}(iii)
shows the sets $I_n \times \{ 0 \} \subset \mathbb{R}^2$ when $\tau_L = 1.2$.
Specifically, we show $f_R^i \left( I_n \times \{ 0 \} \right)$ for all $i = 0,1,\ldots,n-1$,
using a different colour for each $n \in D[5,12;2]$.
The attractor of $f$, shown in Fig.~\ref{fig:bifSet},
appears to be dense on the union of these sets, except covers only part of $I_1 \times \{ 0 \}$.
In this way the configuration of $g$
conveys considerable information about the geometry of the attractor of $f$.

\subsection{Proofs of Lemmas \ref{le:xPrimeLe1} and \ref{le:InInterval}}
\label{sub:lemmaProofs}

\begin{proof}[Proof of Lemma \ref{le:xPrimeLe1}]
If $x \le 0$, then $N(x) = 1$.
In this case $g(x) = \tau_L x + 1 \le 1$, because $\tau_L \ge 0$.

If $x > 0$, then $N(x) \ge 2$ because $\delta_R \ne 0$.
In this case $f^{n-1}(x,0) \in \overline{\Omega}_L$ and $f^{n-2}(x,0) \in \Omega_R$, where $n = N(x)$.
Write $f^{n-1}(x,0) = \left( \tilde{x}, \tilde{y} \right)$
and $f^{n-2}(x,0) = \left( \tilde{\tilde{x}}, \tilde{\tilde{y}} \right)$.
So $\tilde{x} \le 0$ and $\tilde{\tilde{x}} > 0$,
the latter inequality implying $\tilde{y} = -\delta_R \tilde{\tilde{x}} < 0$, because $\delta_R > 0$.
Then $g(x) = \tau_L \tilde{x} + \tilde{y} + 1 \le 1$, because $\tau_L \ge 0$.
\end{proof}

\begin{proof}[Proof of Lemma \ref{le:InInterval}]
If $\delta_R = 0$, then $I_1 = (-\infty,1]$ and $I_n = \varnothing$ for all $n \ge 2$,
so the result holds.

Suppose $\delta_R \ne 0$.
Then $I_1 = (-\infty,0]$.
Choose any $n \ge 2$ and $w_1, w_2 \in I_n$, where $I_n \ne \varnothing$.
Then for each $j = 1,2$, $f^i(w_j,0) \in \Omega_R$ for all $i = 0,1,\ldots,n-2$,
and $f^{n-1}(w_j,0) \in \overline{\Omega}_L$.
Now consider the convex combination $x = (1-s) w_1 + s w_2$, where $s \in [0,1]$.
Since $f$ is affine on the convex set $\Omega_R$,
by induction on $i$ we have $f^i(x,0) \in \Omega_R$ for all $i = 0,1,\ldots,n-2$,
and $f^i(x,0) = (1-s) f^i(w_1,0) + s f^i(w_2,0)$ for all $i = 0,1,\ldots,n-1$.
Thus $f^{n-1}(x,0) \in \overline{\Omega}_L$, and hence $x \in I_n$.
Thus $I_n$ is convex, so is either an interval or a singleton.

With the same $w_1$, consider now the line segment
$K_\ee = \left\{ (x,0) \,\middle|\, x \in [w_1-\ee,w_1+\ee] \right\}$, where $\ee > 0$.
Since $f$ is affine on $\Omega_R$,
there exists $\ee > 0$ such that $f^i(K_\ee) \subset \Omega_R$ for all $i = 0,1,\ldots,n-2$,
and $f^{n-1}(K_\ee)$ is a line segment centred at $z$.
Thus either $f^{n-1}(z-\ee,0) \in \overline{\Omega}_L$, implying $w_1 - \ee \in I_n$,
or $f^{n-1}(z+\ee,0) \in \overline{\Omega}_L$, implying $w_1 + \ee \in I_n$.
In either case, if $w_1 < 1$ then $I_n$ is not the singleton $\{ w_1 \}$.
\end{proof}

\section{Farey constructs and rigid rotations}
\label{sec:Farey}

This section provides a brief review of the Farey organisation of rational numbers on $[0,1]$.
For further details, see \cite{GoTr96,HaZh18,HaWr08,LaTr95}.

A fraction $\frac{m}{p}$ is {\em irreducible} if the integers $m$ and $p$ have no common factors,
and the denominator of an irreducible fraction is always assumed to be positive.
In this section we also use irreducible fractions
to define permutations and relate these to rigid rotations.

\subsection{Farey neighbours and parents}
\label{sub:neighboursParents}

\begin{definition}
Let $\frac{m^-}{p^-}, \frac{m^+}{p^+} \in [0,1]$ be irreducible with $\frac{m^-}{p^-} < \frac{m^+}{p^+}$.
We say $\frac{m^-}{p^-}$ and $\frac{m^+}{p^+}$ are {\em Farey neighbours} if $m^+ p^- - m^- p^+ = 1$.
\end{definition}

\begin{definition}
The {\em Farey sum} of Farey neighbours $\frac{m^-}{p^-}$ and $\frac{m^+}{p^+}$ is
\begin{equation}
\frac{m^-}{p^-} \oplus \frac{m^+}{p^+} = \frac{m^- + m^+}{p^- + p^+}.
\label{eq:FareySum}
\end{equation}
\label{df:FareySum}
\end{definition}

Notice that we only apply the Farey sum to Farey neighbours.
This ensures that the Farey sum enjoys the properties listed below.

\begin{proposition}
For any Farey sum \eqref{eq:FareySum},
\begin{enumerate}[label=\roman*),ref=\roman*,itemsep=0mm]
\item
the right-hand side of \eqref{eq:FareySum} is irreducible,
\item
$\frac{m^-}{p^-}$ and $\frac{m}{p}$ are Farey neighbours,
\item
$\frac{m}{p}$ and $\frac{m^+}{p^+}$ are Farey neighbours.
\end{enumerate}
\label{pr:FareySumIrreducible}
\end{proposition}

\begin{proposition}
For any irreducible $\frac{m}{p} \in (0,1)$,
there exist unique Farey neighbours $\frac{m^-}{p^-} < \frac{m^+}{p^+}$
such that $\frac{m^-}{p^-} \oplus \frac{m^+}{p^+} = \frac{m}{p}$.
Moreover,
\begin{equation}
\begin{aligned}
m^- &= \frac{m d - 1}{p}, & \qquad \qquad
m^+ &= \frac{m (p-d) + 1}{p}, \\
p^- &= d, & \qquad \qquad
p^+ &= p-d,
\end{aligned}
\label{eq:FareyParents}
\end{equation}
where $d$ is the multiplicative inverse of $m$ modulo $p$.
\label{pr:FareyParents}
\end{proposition}

For example, if $\frac{m}{p} = \frac{3}{7}$, then $d = 5$.
Then \eqref{eq:FareyParents} gives $\frac{m^-}{p^-} = \frac{2}{5}$
and $\frac{m^+}{p^+} = \frac{1}{2}$.

\begin{definition}
Let $\frac{m}{p} \in (0,1)$ be irreducible.
The fractions $\frac{m^-}{p^-}$ and $\frac{m^+}{p^+}$ of Proposition \ref{pr:FareyParents}
are the {\em left parent} and {\em right parent} of $\frac{m}{p}$, respectively.
\label{df:FareyParents}
\end{definition}

\subsection{The Farey web and denominator list}
\label{sub:web}

The following terminology is taken from Brucks {\em et al.}~\cite{BrRi02}.

\begin{definition}
The {\em Farey web} is the graph with vertices $\mathbb{Q} \cap [0,1]$
connected by an edge if and only if they are Farey neighbours.
\label{df:FareyWeb}
\end{definition}

Fig.~\ref{fig:sternBrocot} shows the Farey web for numbers
in $\left[ 0, \frac{1}{2} \right]$ up to level six.
We now define the list $D$ used in Theorem \ref{th:existence}.

\begin{definition}
Let $n^* \ge 1$ and $\frac{m}{p} \in \left( 0, \frac{1}{n^*} \right)$ be irreducible.
Let $\frac{m^-_1}{p^-_1}$ be the left parent of $\frac{m}{p}$,
let $\frac{m^-_2}{p^-_2}$ be the left parent of $\frac{m^-_1}{p^-_1}$,
and so on until $\frac{m^-_s}{p^-_s} = \frac{0}{1}$, for some $s \ge 1$.
Let $\frac{m^+_1}{p^+_1}$ be the right parent of $\frac{m}{p}$,
let $\frac{m^+_2}{p^+_2}$ be the right parent of $\frac{m^+_1}{p^+_1}$,
and so on until $\frac{m^+_t}{p^+_t} = \frac{1}{n^*}$, for some $t \ge 1$.
The {\em denominator list} is
\begin{equation}
D[m,p;n^*] = \left( p^-_s, \ldots, p^-_2, p^-_1, p, p^+_1, p^+_2, \ldots, p^+_t \right).
\label{eq:D}
\end{equation}
\label{df:denominatorList}
\end{definition}

\subsection{Permutations corresponding to rigid rotations}
\label{sub:permutations}

We now use rigid rotation with rotation number $\frac{m}{p}$
to define a permutation of $\mathbb{Z}_p = \{ 0,1,\ldots,p-1 \}$.
We then show that the difference
between consecutive elements of this permutation obey a certain minimal property (Lemma \ref{le:differences}).

\begin{definition}
A {\em permutation} of $\mathbb{Z}_p$ is a bijection $\sigma : \mathbb{Z}_p \to \mathbb{Z}_p$,
and we write
\begin{equation}
\sigma = \big( \sigma(0), \sigma(1), \ldots, \sigma(p-1) \big).
\nonumber
\end{equation}
\label{df:permutation}
\end{definition}

\begin{definition}
For any permutation $\sigma$ of $\mathbb{Z}_p$,
we define the {\em difference vector} $v = \partial \sigma$ by
\begin{equation}
v_i = \sigma((i+1) \,{\rm mod}\, p) - \sigma(i),
\label{eq:vi}
\end{equation}
for all $i = 0,1,\ldots,p-1$.
\label{df:differenceVector}
\end{definition}

For example, $\sigma = (3,2,6,1,5,0,4)$ is a permutation of $\mathbb{Z}_7$.
In this case \eqref{eq:vi} gives $\partial \sigma = (-1,4,-5,4,-5,4,-1)$.
Notice that the entries of any difference vector sum to zero.

\begin{definition}
For any irreducible $\frac{m}{p} \in (0,1)$, we define the {\em rotation permutation} $\cP_{m,p}$ by
\begin{equation}
\cP_{m,p}(i m \text{\,mod\,} p) = i, \qquad \text{for all $i=0,1,\ldots,p-1$}.
\label{eq:cP}
\end{equation}
\label{df:rotationPermutation}
\end{definition}

For example, consider $\frac{m}{p} = \frac{3}{7}$,
In this case \eqref{eq:cP} gives $\cP_{3,7}(0) = 0$, 
$\cP_{3,7}(3) = 1$, $\cP_{3,7}(6) = 2$, $\cP_{3,7}(2) = 3$, and so on,
hence $\cP_{3,7} = (0,5,3,1,6,4,2)$.

The rotation permutation $\cP_{m,p}$ can be constructed geometrically as follows.
Draw $p$ nodes on a circle and label one of them $0$, Fig.~\ref{fig:rotWordSchem}(i).
From node $0$, step $m$ nodes clockwise and label it $1$, Fig.~\ref{fig:rotWordSchem}(ii).
Continue stepping $m$ nodes clockwise and labelling $2$, $3$, etc, until all nodes have been labelled, Fig.~\ref{fig:rotWordSchem}(iii).
Then $\cP_{m,p}$ is given by the labels of the nodes ordered clockwise from $0$.

\begin{figure}[t!]
\begin{center}
\includegraphics[width=15cm]{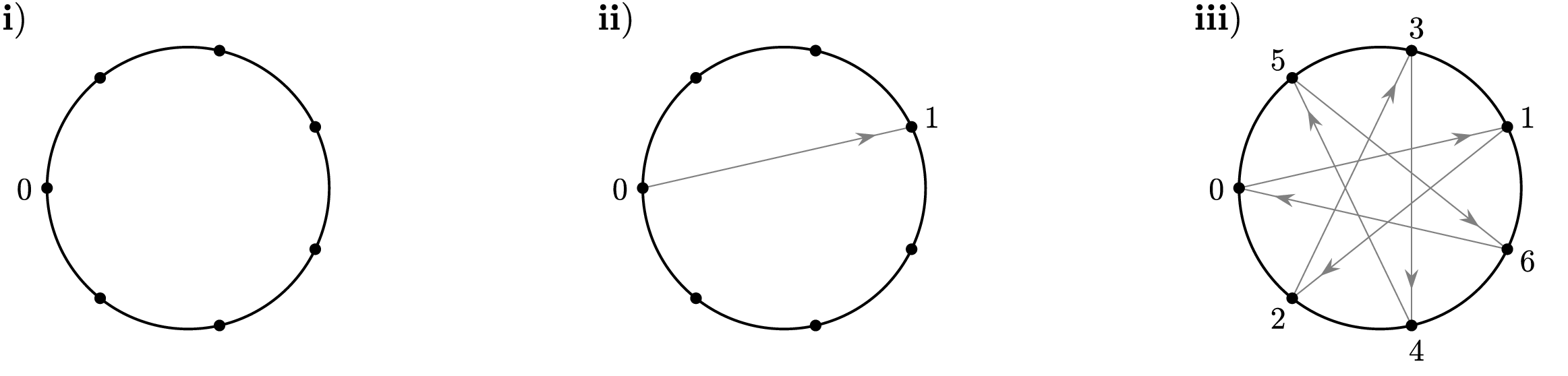}
\caption{
A illustration of the geometric construction of $\cP_{3,7}$ described in the text.
\label{fig:rotWordSchem}
} 
\end{center}
\end{figure}

The difference vector of $\sigma = \cP_{3,7}$ is $\partial \sigma = (5,-2,-2,5,-2,-2,-2)$.
Notice this vector contains only two distinct values.
The following result shows that this property can be used to characterise $\cP_{m,p}$.

\begin{lemma}
Let $\frac{m}{p} \in (0,1)$ be irreducible
with left parent $\frac{m^-}{p^-}$ and right parent $\frac{m^+}{p^+}$,
and let $\sigma$ be a permutation of $\mathbb{Z}_p$.
Then $\sigma = \cP_{m,p}$ if and only if $\sigma(0) = 0$ and
every entry of $\partial \sigma$ is either $p^-$ or $-p^+$.
\label{le:differences}
\end{lemma}

\begin{proof}
Suppose $\sigma = \cP_{m,p}$.
By \eqref{eq:FareyParents}, $p^-$ is the multiplicative inverse of $m$ modulo $p$,
thus by substituting $j = i m \text{~mod~} p$ into \eqref{eq:cP} we obtain
$\sigma(j) = j p^- \text{~mod~} p$ for all $j$.
So $\sigma(0) = 0$, and as we go from one index to the next the value either increases by $p^-$,
or increases by $p^-$ but exceeds $p$, so actually decreases by $p - p^- = p^+$, due the modulo $p$.
Thus every entry of $\partial \sigma$ is either $p^-$ or $-p^+$.

There is only one permutation of $\mathbb{Z}_p$
with $\sigma(0) = 0$ and for which
every entry of $\partial \sigma$ is either $p^-$ or $-p^+$.
This is because the permutation starts from $0$,
and then as we go from one index to the next there is only ever one possibility
for whether the value increases by $p^-$ or decreases by $p^+$
because $p^- + p^+ = p$ and the values of the permutation are constrained to the set $\mathbb{Z}_p$.
\end{proof}

\section{Main arguments}
\label{sec:arguments}

In this section we prove Theorems \ref{th:existence} and \ref{th:construct}.
First in \S\ref{sub:smp}
we map the switching manifold of $f$ backwards under $f_R$,
and relate it to the fixed point $X$ of $f_R$.
Then in \S\ref{sub:int} we
show that points where these preimages intersect $\Sigma_1$
are exactly the points $(z_n,0)$ obtained in \S\ref{sub:lemmaProofs} by mapping forwards under $f_R$.

Recall, $r = N(x) - 1$ is the number of iterations
required for the forward orbit of $(x,0)$ to enter $\overline{\Omega}_L$.
In \S\ref{sub:partition} we extend this definition to all points $(x,y) \in \mathbb{R}^2$,
and study the regions $E_r$ of constant $r$.
Together with the fixed point $X$, these sets partition the $(x,y)$-plane.
We also introduce the complement sets
$F_r = \mathbb{R}^2 \setminus (E_0 \cup E_1 \cup \cdots \cup E_{r-1})$.

Recall, $I_1 = (-\infty,0]$ and $1 \in I_{n^*}$.
In \S\ref{sub:FnStarPlus1} we characterise $I_{n^*}$
and show that $I_n = \varnothing$ for all $n = 2,3,\ldots,n^*-1$.
This is achieved by meticulously characterising $E_r$ for all $r \le n^*$.
For all $r \ge n^*+1$, $F_r$ is a polygon,
and in \S\ref{sub:mpForm} we say that it has `$\frac{m}{p}$-form' if the ordering
of its edges according to the action of $f_R^{-1}$
conforms to the rotation permutation $\cP_{m,p}$
(and some additional properties are satisfied, see Definition \ref{df:mpForm}).

In \S\ref{sub:threeCases} we take steps down the Farey web,
and show that at each step $F_p$ has $\frac{m}{p}$-form.
This enables us to prove $I_n = \varnothing$ for all values of $n$
between one value of $p$ and the next.
In \S\ref{sub:final} we complete the proof of Theorem \ref{th:existence},
then prove Theorem \ref{th:construct} by showing that our
steps are identical to those defined by the algorithm in Theorem \ref{th:construct}.

\subsection{The fixed point and switching manifold preimages}
\label{sub:smp}

If $\delta_R - \tau_R + 1 \ne 0$, then $f_R$ has the unique fixed point
\begin{equation}
X = \left( \frac{1}{\delta_R - \tau_R + 1}, \frac{-\delta_R}{\delta_R - \tau_R + 1} \right).
\label{eq:X}
\end{equation}
If also $\delta_R - \tau_R + 1 \ne 0 > 0$,
then $X \in \Omega_R$, and so $X$ is a fixed point of $f$.

The switching manifold of $f$ is the $y$-axis
\begin{equation}
\Sigma_0 = \left\{ (x,y) \in \mathbb{R}^2 \,\middle|\, x = 0 \right\},
\label{eq:Sigma0}
\end{equation}
and its image under $f$ is the $x$-axis, $\Sigma_1$.
Now suppose $\delta_R \ne 0$, so that $f_R$ is invertible.
Then for all $r \ge 1$,
\begin{equation}
\Sigma_{-r} = f_R^{-r}(\Sigma_0),
\label{eq:SigmaMinusr}
\end{equation}
is a line in the $(x,y)$-plane.

The sets
\begin{align}
U_L^r &= f_R^{-r}(\Omega_L), &
U_R^r &= f_R^{-r}(\Omega_R),
\label{eq:UrLUrRdefn}
\end{align}
are half-planes with boundary $\Sigma_{-r}$.
If $\delta_R - \tau_R + 1 > 0$, then $X \in \Omega_R$, and
\begin{align}
X &\in U_R^r \,, &
X &\notin U_L^r \,,
\label{eq:UrRUrL}
\end{align}
for all $r \ge 0$.

Let $s_r \in \mathbb{R} \cup \{ \infty \}$ denote the slope $\frac{dy}{dx}$ of $\Sigma_{-r}$.
From the formula \eqref{eq:fLfR} for $f_R$, it is a simple exercise
to derive the following recurrence relation (given also in \cite{SiGl24}).

\begin{lemma}
Suppose $\delta_R \ne 0$.
Then $s_0 = \infty$ and
\begin{equation}
s_r = \begin{cases}
\infty, & s_{r-1} = 0, \\
-\tau_R \,, & s_{r-1} = \infty, \\
\frac{-\delta_R}{s_{r-1}} - \tau_R \,, & \text{otherwise},
\end{cases}
\label{eq:slopeRecurrenceRelation}
\end{equation}
for all $r \ge 1$.
\end{lemma}

\begin{figure}[t!]
\begin{center}
\includegraphics[width=4.8cm]{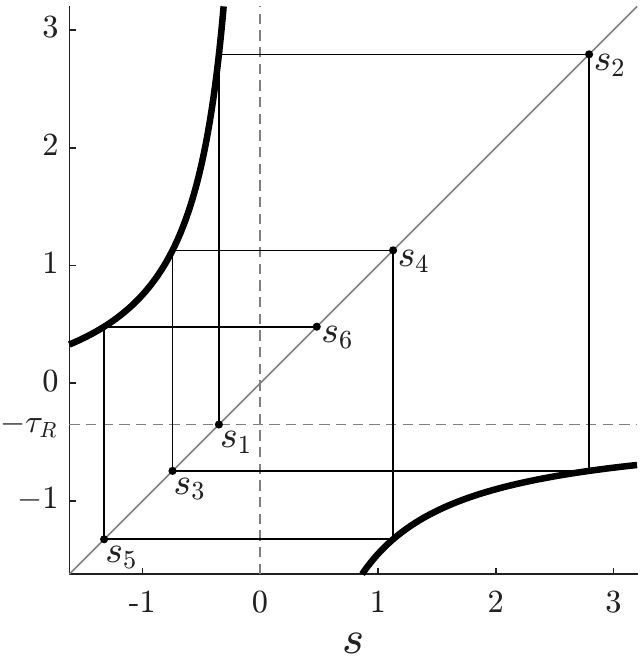}
\caption{
A cobweb diagram of the slope map \eqref{eq:slopeRecurrenceRelation}
with $\tau_R = 0.35$ and $\delta_R = 1.1$,
corresponding to point {\bf c} of Fig.~\ref{fig:bifSet}.
\label{fig:slopeMap}
} 
\end{center}
\end{figure}

The recurrence relation \eqref{eq:slopeRecurrenceRelation} is a map illustrated in Fig.~\ref{fig:slopeMap}.
In the case $\delta_R > \frac{\tau_R^2}{4}$,
we obtain from \eqref{eq:slopeRecurrenceRelation} the explicit formula
\begin{equation}
s_r = \frac{-\sqrt{\delta_R} \,\sin((r+1) \phi)}{\sin(r \phi)},
\label{eq:sr1}
\end{equation}
where $\phi$ is given by \eqref{eq:phi}.

Recall from \eqref{eq:fRr1} that $\alpha_r$ is the $x$-coefficient of $f_R^r(x,0)_1$.
The following result expresses $s_r$ in terms of these coefficients,
and is an immediate consequence of \eqref{eq:alphabeta} and \eqref{eq:sr1}.

\begin{lemma}
Suppose $\delta_R > \frac{\tau_R^2}{4}$.
Then
\begin{equation}
s_r = \begin{cases}
\infty, & \alpha_{r-1} = 0, \\
-\frac{\alpha_r}{\alpha_{r-1}}, & \text{otherwise},
\end{cases}
\label{eq:sr2}
\end{equation}
for all $r \ge 1$.
\label{le:sr2}
\end{lemma}

\subsection{Intersections of switching manifold preimages}
\label{sub:int}

If $s_r \ne 0$, then $\Sigma_{-r}$ and $\Sigma_1$
are lines with different slopes, so intersect at a unique point.
Also $s_r \ne 0$ implies $\alpha_r \ne 0$,
so the value $z_n$ \eqref{eq:zn}, where $n = r-1$, is well-defined.

\begin{lemma}
Suppose $\delta_R \ne 0$, and let $n \ge 1$.
If $\alpha_{n-1} \ne 0$, then
\begin{equation}
\Sigma_{-(n-1)} \cap \Sigma_1 = \{ (z_n,0) \}.
\label{eq:znAsIntersection}
\end{equation}
\label{le:znAsIntersection}
\end{lemma}

\begin{proof}
As noted above, $\Sigma_{-(n-1)}$ and $\Sigma_1$ intersect at a unique point.
By definition, $(z_n,0) \in \Sigma_1$ maps under $f_R^{n-1}$ to $\Sigma_0$,
so $(z_n,0) \in \Sigma_{-(n-1)}$.
\end{proof}

If $\alpha_{n-1} \ne 0$, write
\begin{equation}
Q_n^j = f_R^{-j}(z_n,0),
\label{eq:Qnj}
\end{equation}
for all $j \ge 0$.
Notice
\begin{equation}
Q_n^1 = (0,z_n - 1),
\label{eq:Qn1}
\end{equation}
in view of the formula \eqref{eq:fLfR} for $f_R$.
By iterating the elements of \eqref{eq:znAsIntersection} under $f_R^{-1}$,
we can characterise the intersection of any two switching manifold preimages.
Here we assume $f_R$ corresponds to a repelling focus
so that switching manifold preimages cannot coincide.

\begin{lemma}
Suppose $\delta_R > 1$ and $\delta_R > \frac{\tau_R^2}{4}$. 
Let $j \ge 0$ and $n \ge 1$.
The lines $\Sigma_{-j+1}$ and $\Sigma_{-n-j+1}$ intersect if and only if $\alpha_{n-1} \ne 0$.
If $\alpha_{n-1} \ne 0$, then
\begin{equation}
\Sigma_{-(n+j-1)} \cap \Sigma_{-(j-1)} = \left\{ Q_n^j \right\}.
\label{eq:intersectionFormula}
\end{equation}
\label{le:intersectionFormula}
\end{lemma}

\begin{proof}
If $\alpha_{n-1} \ne 0$,
then \eqref{eq:intersectionFormula} follows immediately from \eqref{eq:znAsIntersection}
by iterating $j$ times under $f_R^{-1}$.
If $\alpha_{n-1} = 0$, then $s_{n-1} = 0$ (see Lemma \ref{le:sr2}),
so $\Sigma_{-(n-1)}$ is parallel to $\Sigma_1$,
and hence $\Sigma_{-(n+j-1)}$ is parallel to $\Sigma_{-(j-1)}$.
These lines cannot coincide because $f_R$ is affine
with a repelling focus fixed point that does not belong to the switching manifold,
thus $\Sigma_{-(n+j-1)} \cap \Sigma_{-(j-1)} = \varnothing$.
\end{proof}

\subsection{A partition of the plane}
\label{sub:partition}

Given $(x,y) \in \mathbb{R}^2$,
let $\chi(x,y)$ be the smallest $r \ge 0$
for which $f^r(x,y) \in \overline{\Omega}_L$.
For points with $y = 0$, this connects to the definition of the return time $N(x)$.
Specifically, if $\delta_R \ne 0$, then
\begin{equation}
\chi(x,0) = N(x) - 1,
\label{eq:chiToN}
\end{equation}
for all $x \in \mathbb{R}$.

As in \cite{SiGl24}, we consider the regions
\begin{equation}
E_r = \left\{ (x,y) \in \mathbb{R}^2 \,\middle|\, \chi(x,y) = r \right\}.
\label{eq:Er}
\end{equation}
Writing $\tilde{\Sigma}_1$ for the $x \le 1$ part of $\Sigma_1$, we have
\begin{equation}
I_n \times \{ 0 \} = E_{n-1} \cap \tilde{\Sigma}_1 \,,
\label{eq:In2}
\end{equation}
for all $n \ge 1$, by \eqref{eq:chiToN}. 
The intersections \eqref{eq:In2} are indicated Fig.~\ref{fig:divH}
for parameter point {\bf c} of Fig.~\ref{fig:bifSet}.

\begin{figure}[b!]
\begin{center}
\includegraphics[width=9cm]{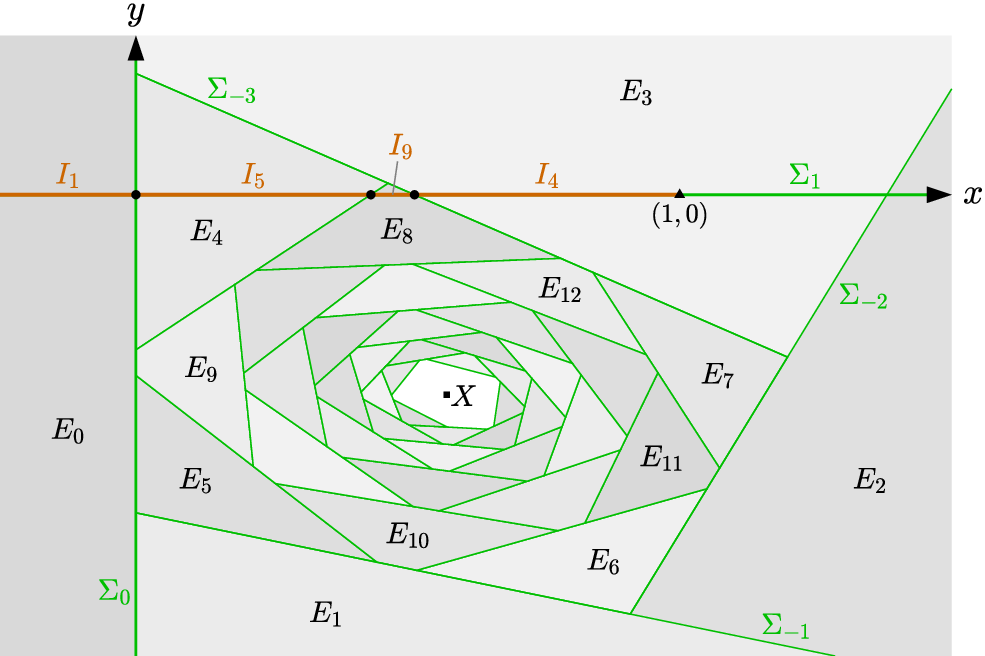}
\caption{
The sets $E_r$ using $\tau_R = 0.35$ and $\delta_R = 1.1$,
which corresponds to point {\bf c} of Fig.~\ref{fig:bifSet}.
We also indicate the non-empty intervals $I_n$,
and the repelling focus fixed point $X$.
\label{fig:divH}
} 
\end{center}
\end{figure}

Now let $F_0 = \mathbb{R}^2$, and
\begin{equation}
F_r = \mathbb{R}^2 \setminus (E_0 \cup E_1 \cup \cdots \cup E_{r-1}),
\label{eq:Fr}
\end{equation}
for all $r \ge 1$.
Notice $F_1 = \Omega_R$ because $E_0 = \overline{\Omega}_L$.
Using \eqref{eq:UrLUrRdefn}, we have the equivalent definition
\begin{equation}
F_r = U_R^0 \cap U_R^1 \cap \cdots \cap U_R^{r-1} \,.
\label{eq:FrU}
\end{equation}
assuming if $\delta_R \ne 0$.

\begin{lemma}
The sets $E_r$ and $F_r$ are convex.
\label{le:convexity}
\end{lemma}

\begin{proof}
By \eqref{eq:FrU}, $F_r$ is the intersection of convex sets, thus is convex.
Moreover, $E_r$ is the intersection of $F_r$ and the closure of $U_L^r$, so is convex.
\end{proof}

The next result shows that each $E_r$ and $F_r$
result from mapping $E_{r-1}$ and $F_{r-1}$ backwards under $f_R$,
after first removing all points on or above $\Sigma_1$.

\begin{lemma}
Suppose $\delta_R > 0$.
Then
\begin{align}
E_r &= \left\{ f_R^{-1}(P) \,\middle|\, P \in E_{r-1}, P_2 < 0 \right\}, \label{eq:Er4} \\
F_r &= \left\{ f_R^{-1}(P) \,\middle|\, P \in F_{r-1}, P_2 < 0 \right\}, \label{eq:Fr4}
\end{align}
for all $r \ge 1$.
\label{le:ErFr}
\end{lemma}

\begin{proof}
We just derive \eqref{eq:Er4},
as \eqref{eq:Fr4} can be obtained in the same fashion.
By definition, $E_r$ is the set of all $Z \in \mathbb{R}^2$
for which $Z, f_R(Z), \ldots, f_R^{r-1}(Z) \in \Omega_R$,
and $f_R^r(Z) \in \overline{\Omega}_L$.
Notice $Z \in \Omega_R$ is equivalent to $Z_1 > 0$,
and write $P = f_R(Z)$.
Then $E_r$ is the set of all $Z \in \mathbb{R}^2$
for which $Z_1 > 0$ and $P, \ldots, f_R^{r-2}(P) \in \Omega_R$,
and $f_R^{r-1}(P) \in \overline{\Omega}_L$.
That is, $Z_1 > 0$ and $P \in E_{r-1}$.
But $Z_1 > 0$ is equivalent to $P_2 < 0$, because $\delta_R > 0$, thus we have \eqref{eq:Er4}.
\end{proof}

We now show that the sets $E_r$, together with $\{ X \}$, form a partition of the $(x,y)$-plane.

\begin{lemma}
Suppose $\delta_R > \frac{\tau_R^2}{4}$.
The sets $\{ X \}, E_0, E_1, \ldots$ are all non-empty and form a disjoint union of $\mathbb{R}^2$.
\label{le:partition}
\end{lemma}

\begin{proof}
The sets $E_r$ are disjoint by definition.
The forward orbit of $X$ remains at $X$, so does not belong to any set $E_r$.
The forward orbit of other point in $\mathbb{R}^2$
escapes $\Omega_R$, because $f_R$ is affine and $X$ is a repelling focus, thus belongs to some set $E_r$.
Thus $\{ X \}, E_0, E_1, \ldots$ form a disjoint union of $\mathbb{R}^2$.

Suppose for a contradiction $E_r = \varnothing$ for some $r \ge 0$.
By the continuity of $f_R$ at $X$,
there exists $P \in \mathbb{R}^n$ (near $X$) such that $P \in E_s$ with some $s > r$.
But then $f_R^{s-r}(P) \in E_r$, which is a contradiction.
\end{proof}

\subsection{A characterisation of $I_{n^*}$ and $F_{n^*+1}$}
\label{sub:FnStarPlus1}

Recall, $n^* = N(1)$, by definition, \eqref{eq:nStar}.
So $\chi(1,0) = n^* - 1$, by \eqref{eq:chiToN},
and hence $(1,0) \in E_{n^*-1}$.
Thus, by \eqref{eq:In2}, $I_{n^*}$ consists of all $x \in (0,1]$
for which $(x,0) \in E_{n^*-1}$.

By characterising the sets $E_1, E_2, \ldots, E_{n^*-2}$,
we are now able to show that $I_2 = I_3 = \cdots = I_{n^*-1} = \varnothing$, Proposition \ref{pr:nStar}(i).
By then characterising $E_{n^*-1}$, we are able to prove that $z_{n^*}$ is well-defined
and is the left-most point of $I_{n^*}$, Proposition \ref{pr:nStar}(ii).
By also characterising $E_{n^*}$,
we can show via \eqref{eq:Fr} that $F_{n^*+1}$ is a polygon with certain properties,
Proposition \ref{pr:nStar}(iii).

\begin{proposition}
Suppose $\delta_R > 1$ and $\delta_R > \frac{\tau_R^2}{4}$.
\begin{enumerate}[label=\roman*),ref=\roman*,itemsep=0mm]
\item 
$I_n = \varnothing$ for all $n = 2,3,\ldots,n^*-1$.
\item
The value $z_{n^*}$ is well-defined, belongs to $(0,1]$, and $I_{n^*} = \left[ z_n^*, 1 \right]$.
\item
The set $F_{n^*+1}$ is the interior of the convex polygon
with vertices $Q_1^1, Q_1^2, \ldots, Q_1^{n^*}, Q_{n^*}^1$ ordered anticlockwise.
The vertices $Q_1^1, Q_1^2, \ldots, Q_1^{n^*-1}$ lie below $\Sigma_1$;
the vertex $Q_{n^*}^1$ lies on or below $\Sigma_1$.
\end{enumerate}
\label{pr:nStar}
\end{proposition}

\begin{figure}[b!]
\begin{center}
\includegraphics[width=15cm]{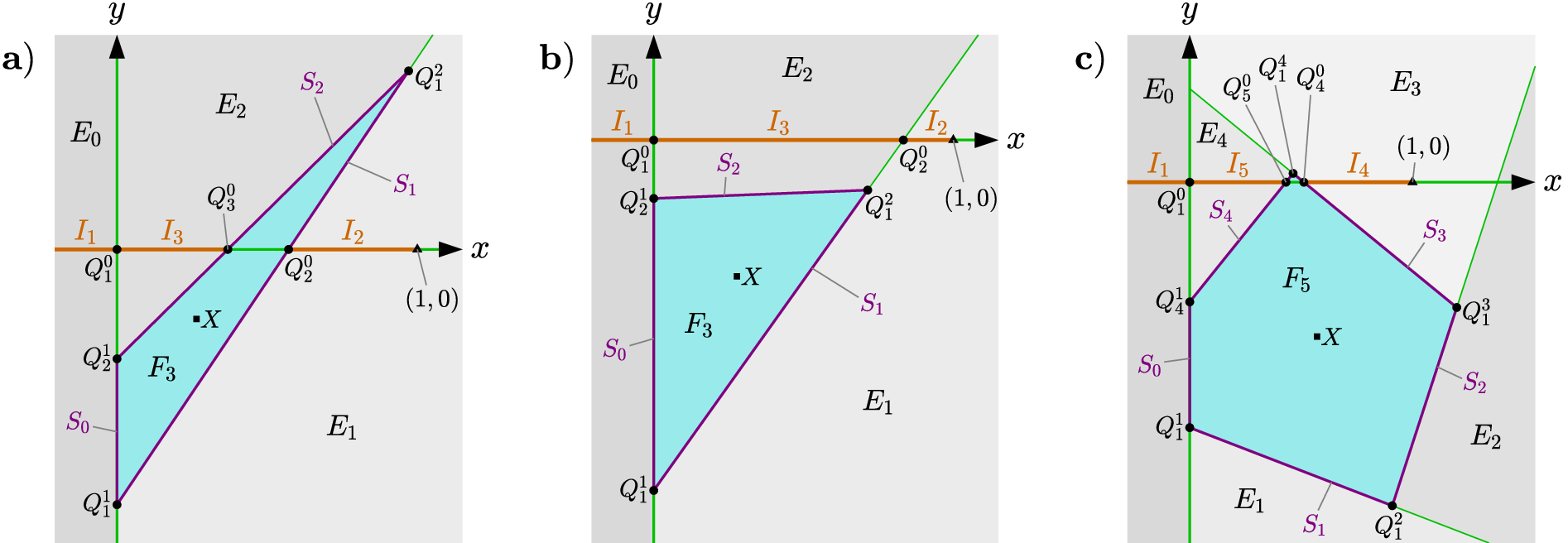}
\caption{
The polygons $F_{n^*+1}$ for parameter points {\bf a}, {\bf b}, and {\bf c} of Fig.~\ref{fig:bifSet}.
In {\bf a} and {\bf b}, $n^* = 2$; in {\bf c}, $n^* = 4$.
The sides of $F_{n^*+1}$ are labelled $S_0, S_1, \ldots, S_{n^*}$,
where each $S_j$ is a subset of $\Sigma_{-j}$.
We also indicate the sets $E_0, E_1, \ldots, E_{n^*}$,
and the intervals $I_1$, $I_{n^*}$, and $I_{n^*+1}$.
\label{fig:divK}
} 
\end{center}
\end{figure}

The polygon $F_{n^*+1}$ is shown in Fig.~\ref{fig:divK}
for our three main examples, {\bf a}, {\bf b}, and {\bf c},
for which $n^* = 2$, $n^* = 2$, and $n^* = 4$, respectively.
Notice Proposition \ref{pr:nStar} does not provide any information about
the position of $Q_1^{n^*}$ relative to $\Sigma_1$.
At parameter point {\bf b}, $Q_1^{n^*}$ lies below $\Sigma_1$.
In this case $F_{n^*+1}$ lies entirely below $\Sigma_1$,
thus $I_n = \varnothing$ for all $n > n^*+ 1$.
At parameter points {\bf a} and {\bf c}, $Q_1^{n^*}$ lies above $\Sigma_1$.
In this case $F_{n^*+1}$ contains points on $\Sigma_1$ with $x < 1$,
thus $I_n \ne \varnothing$ for some $n > n^*+ 1$.

Proposition \ref{pr:nStar} is fairly intuitive given the rotating nature of $f_R$.
The sequence of switching manifold preimages $\Sigma_{-r}$ essentially
revolves anticlockwise around $X$, as seen most clearly for parameter point {\bf c}.
Our proof of Proposition \ref{pr:nStar} is long and notation-heavy
because we find it necessary to characterise each set $E_1, E_2, \ldots, E_{n^*}$,
some of which are unbounded.
For these reasons the proof is deferred to Appendix \ref{app:nStarProof}.

\subsection{The $\frac{m}{p}$-form}
\label{sub:mpForm}

The sets $F_r$ are nested: by Lemma \ref{le:partition},
each $F_r$ is a proper subset of $F_{r-1}$.
As $r \to \infty$, the $F_r$ converge to the fixed point $X$.
Fig.~\ref{fig:divN} shows $F_r$ for $r = 1,2,\ldots,16$ for parameter point {\bf a}.
Notice $F_r$ is an $r$-sided polygon for $r = 3,4,\ldots,12$.
Notice also that $F_{12}$ lies below $\Sigma_1$,
so, by \eqref{eq:Fr4}, we have $f_R^{-1}(F_{12}) = F_{13}$,
and thus $f_R^{-1}(F_{13}) = F_{14}$, and so on. 
Hence $F_r$ is a $12$-sided polygon for all $r \ge 12$.

\begin{figure}[b!]
\begin{center}
\includegraphics[width=15cm]{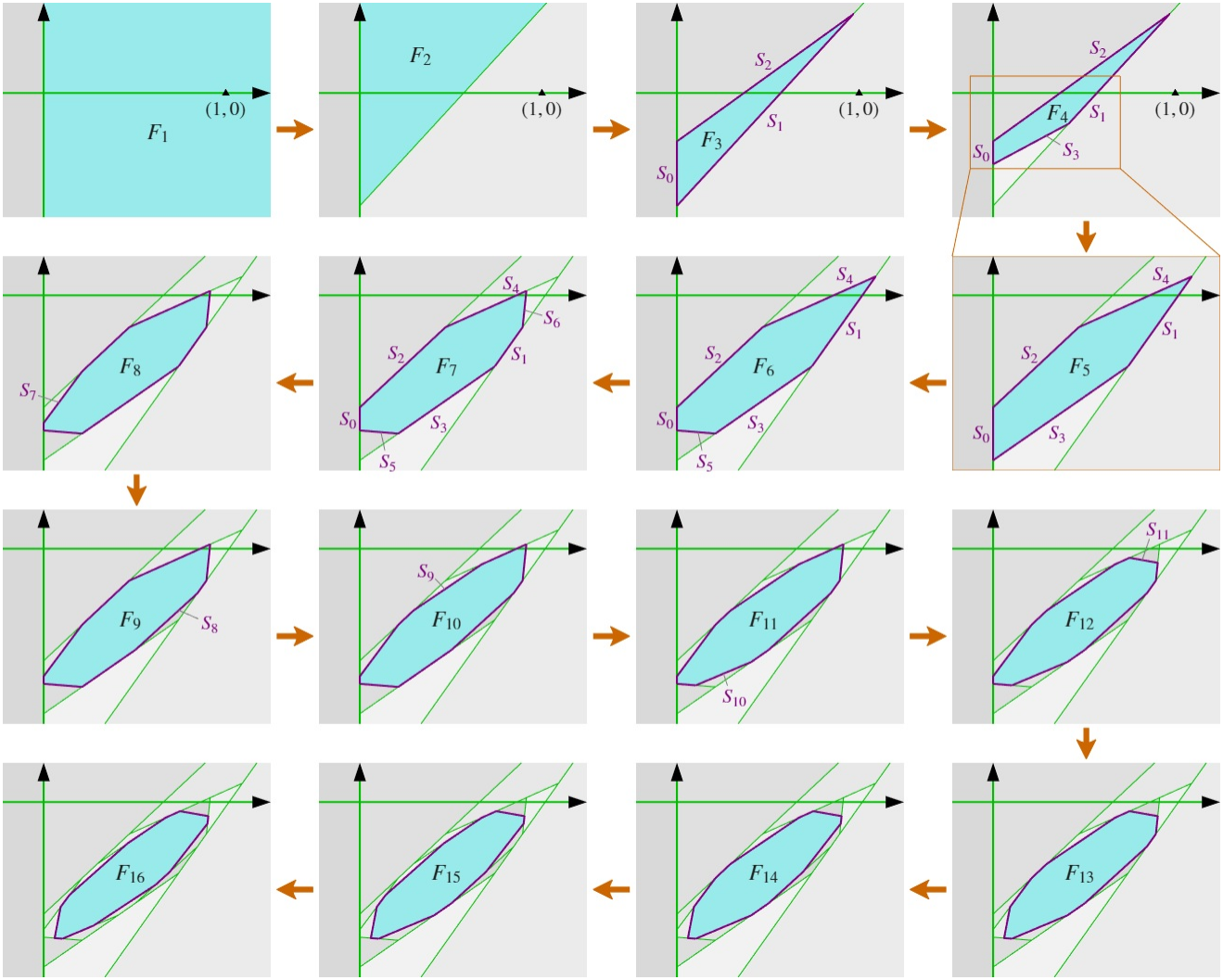}
\caption{
The sets $F_1, F_2, \ldots, F_{16}$ for parameter point {\bf a} of Fig.~\ref{fig:bifSet}.
\label{fig:divN}
} 
\end{center}
\end{figure}

In general, the first few $F_r$ are unbounded,
the next few $F_r$ are $r$-sided polygons with one side $S_0 \subset \Sigma_0$,
and the remaining $F_r$ are polygons with a fixed number of edges.
So, as $r$ increases, we can think of the sets $F_r$
as evolving with three stages of development.
The following definition refers to the middle stage.

\begin{definition}
We say $F_r$ is {\em adolescent} if it is an $r$-sided polygon
and its edges $S_0,S_1,\ldots,S_{r-1}$ can be enumerated in such a way
that $S_j \subset \Sigma_{-j}$ for each $j$.
We write $\nu(F_r)$ for the ordering of the $S_j$
starting from $S_0$ and going anticlockwise around $F_r$.
\label{df:adolescent}
\end{definition}

For example, in Fig.~\ref{fig:divN},
$\nu(F_5) = (0,3,1,4,2)$,
$\nu(F_6) = (0,5,3,1,4,2)$, and
$\nu(F_7) = (0,5,3,1,6,4,2)$.
In Definition \ref{df:adolescent} the ordering is anticlockwise,
instead of clockwise as in Fig.~\ref{fig:rotWordSchem},
because the edges $S_j$ result from mapping $\Sigma_0$ backwards under $f_R$.

\begin{definition}
Suppose $\delta_R > 1$ and $\delta_R > \frac{\tau_R^2}{4}$,
and let $\frac{m}{p} \in \left( 0, \frac{1}{2} \right)$ be irreducible.
We say $F_p$ has {\em $\frac{m}{p}$-form} if
\begin{enumerate}[label=\roman*),ref=\roman*,itemsep=0mm]
\item
it is adolescent,
\item
its edges $S_j$ are ordered anticlockwise from $S_0$ by $\cP_{m,p}$,
\item
the upper vertex of $S_0$ lies on or below $\Sigma_1$,
and if $F_p$ has a vertex on or above $\Sigma_1$,
then this vertex is an endpoint of $S_{p-1}$.
\end{enumerate}
\label{df:mpForm}
\end{definition}

In Fig.~\ref{fig:divN}, $F_5$ has $\frac{2}{5}$-form, and $F_7$ has $\frac{3}{7}$-form,
while $F_6$ does not have $\frac{m}{6}$-form for any value of $m$.
The constraints in condition (iii) of Definition \ref{df:mpForm}
will aid our proofs of Theorems \ref{th:existence} and \ref{th:construct} below.

\begin{lemma}
$F_{n^*+1}$ has $\frac{1}{n^*+1}$-form.
\label{le:baseCase}
\end{lemma}

\begin{proof}
The left and right Farey parents of $\frac{1}{n^*+1}$ are
$\frac{m^-}{p^-} = \frac{0}{1}$ and $\frac{m^+}{p^+} = \frac{1}{n^*}$, respectively.
By Proposition \ref{pr:nStar},
$F_{n^*+1}$ is the $(n^*+1)$-sided polygon with vertices $Q_1^1, Q_1^2, \ldots, Q_1^{n^*}, Q_{n^*}^1$.
By Lemma \ref{le:intersectionFormula},
the line through the adjacent vertices $Q_{n^*}^1$ and $Q_1^1$ is $\Sigma_0$,
and for each $j = 1,2,\ldots,n^*-1$ the line through the adjacent vertices
$Q_1^j$ and $Q_1^{j+1}$ is $\Sigma_{-j}$, verifying (i).
This also verifies (ii), because $\cP[1,n^*+1] = (0,1,\ldots,n^*)$.
Finally, (iii) follows immediately part (iii) of Proposition \ref{pr:nStar}.
\end{proof}

In the remainder of this section
we write $\frac{m^-}{p^-}$ and $\frac{m^+}{p^+}$, respectively,
for the left and right parents of $\frac{m}{p}$.

\begin{lemma}
Suppose $F_p$ has $\frac{m}{p}$-form.
Then
\begin{enumerate}[label=\roman*),ref=\roman*,itemsep=0mm]
\item
$z_{p^-}$ and $z_{p^+}$ are well-defined with $0 \le z_{p^-} < z_{p^+} \le 1$,
\item
the vertices of $F_p$ are $Q_{p^-}^1, Q_{p^-}^2, \ldots, Q_{p^-}^{p^+}$
and $Q_{p^+}^1, Q_{p^+}^2, \ldots, Q_{p^+}^{p^-}$,
\item
for each $j = 1,2,\ldots,p^+$ the edge $S_{j-1}$ is clockwise from $S_{j + p^- - 1}$ and meets at the vertex $S_{p^-}^j$, and
for each $j = 1,2,\ldots,p^-$ the edge $S_{j-1}$ is anticlockwise from $S_{j + p^+ - 1}$ and meets at the vertex $S_{p^+}^j$.
\end{enumerate}
\label{le:verticesFp}
\end{lemma}

As an example, Fig.~\ref{fig:divO} shows $F_{12}$ of Fig.~\ref{fig:divN},
but now with all of its edges and vertices labelled.
The set $F_{12}$ has $\frac{5}{12}$-form,
and by inspection we can see that it obeys Lemma \ref{le:verticesFp}
with $\frac{m^-}{p^-} = \frac{2}{5}$ and $\frac{m^+}{p^+} = \frac{3}{7}$.

\begin{figure}[b!]
\begin{center}
\includegraphics[width=9cm]{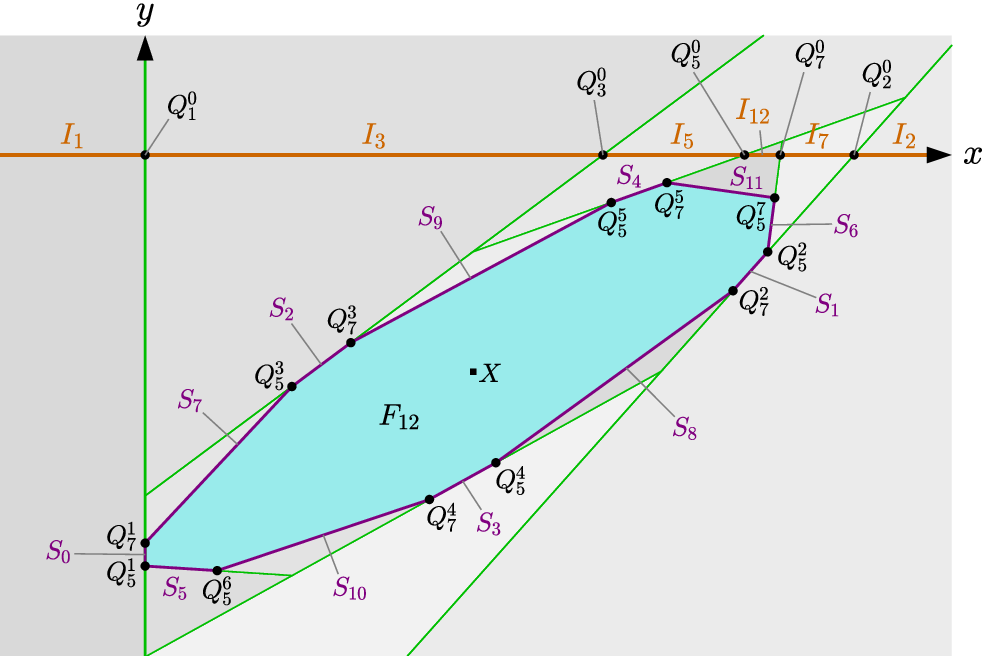}
\caption{
The polygon $F_{12}$ for parameter point {\bf a} of Fig.~\ref{fig:bifSet}.
\label{fig:divO}
} 
\end{center}
\end{figure}

\begin{proof}
Choose any $j \in \{ 0,1,\ldots,p-1 \}$.
Since $F_p$ has $\frac{m}{p}$-form, the edge anticlockwise from $S_j$ is $S_k$,
where $k$ is the value in $\cP[m,p]$ that precedes $j$ (cyclically).
By Lemma \ref{le:differences}, if $j < p^-$ then $k = j + p^+$,
so by Lemma \ref{le:intersectionFormula} these edges meet at $Q_{p^+}^{j + 1}$.
Thus $z_{p^+}$ is well-defined
and $Q_{p^+}^1, Q_{p^+}^2, \ldots, Q_{p^+}^{p^-}$ are vertices of $F_p$.
If instead $j \ge p^-$, then $k = j - p^-$,
so by Lemma \ref{le:intersectionFormula} the edges $S_j$ and $S_k$ meet at $Q_{p^-}^{j - p^- + 1}$.
Thus $z_{p^-}$ is well-defined
and that $Q_{p^-}^1, Q_{p^-}^2, \ldots, Q_{p^-}^{p^+}$
are vertices of $F_p$.

Notice $F_p \subset \Omega_R$ with edge $S_0 \subset \Sigma_0$.
This edge has vertices $Q_{p^-}^1 = \left( 0, z_{p^-} - 1 \right)$
and $Q_{p^+}^1 = \left( 0, z_{p^+} - 1 \right)$, see \eqref{eq:Qn1}.
The edge anticlockwise from $S_0$ is $S_{p^-}$,
thus by \eqref{eq:intersectionFormula} these edges share the vertex $Q_{p^-}^1$.
That is, $Q_{p^-}^1$ is situated below $Q_{p^+}^1$, which lies on or below $\Sigma_1$
by condition (iii) of Definition \ref{df:mpForm}, so $z_{p^-} < z_{p^+} \le 1$.

Suppose for a contradiction $z_{p^-} < 0$.
Then $Q_{p^-}^1 = (0,z_{p^-} - 1)$ lies below $(0,-1)$,
and hence $F_p$ contain a point $(\varepsilon,y)$,
with $\varepsilon > 0$ and $y < -1$, that maps $\Omega_L$.
By definition this point belongs to $E_1$, so cannot belong to $F_p$.
This is a contradiction, hence $z_{p^-} \ge 0$.
\end{proof}

Write
\begin{equation}
J_n = \bigcup_{j \ge n} I_j \,,
\label{eq:J}
\end{equation}
for all $n \ge 1$.
The following result characterises $J_p$ when $F_p$ has $\frac{m}{p}$-form.
This result is useful because it shows that
none of the sets $I_1, I_2, \ldots, I_{p-1}$
contain values in the interval $\left( z_{p^-}, z_{p^+} \right)$.

\begin{lemma}
If $F_p$ has $\frac{m}{p}$-form,
then $J_p = \left( z_{p^-}, z_{p^+} \right)$.
\label{le:Jp}
\end{lemma}

\begin{proof}
By Definition \ref{df:mpForm},
the intersection of $\Sigma_0$ with the closure of $F_p$ is the edge $S_0$.
Then by \eqref{eq:Fr4},
the intersection of $\Sigma_1$ with the closure of $F_{p-1}$ is $f_R(S_0)$.
Notice $J_p \times \{ 0 \} = F_{p-1} \cap \tilde{\Sigma}_1$,
by \eqref{eq:In2} and \eqref{eq:Er}.
So $J_p$ consists of all $x \le 1$ for which $(x,0)$ lies between the endpoints of $f_R(S_0)$.

By Lemma \ref{le:verticesFp}, $S_0$ has endpoints $Q_{p^-}^1$ and $Q_{p^+}^1$.
Thus $f_R(S_0)$ has endpoints $Q_{p^-}^0 = (z_{p^-},0)$ and $Q_{p^+}^0 = (z_{p^+},0)$.
Since $0 \le z_{p^-} < z_{p^+} \le 1$, by Lemma \ref{le:verticesFp}(i),
we have $J_p = \left( z_{p^-}, z_{p^+} \right)$.
\end{proof}

\subsection{Three cases}
\label{sub:threeCases}

Consider the edge $S_{p-1}$ of $F_p$ in $\frac{m}{p}$-form.
By Lemma \ref{le:verticesFp},
the edge clockwise from $S_{p-1}$ is $S_{p^+ - 1}$,
and by \eqref{eq:intersectionFormula} these edges meet at $Q_{p^-}^{p^+}$.
The edge anticlockwise from $S_{p-1}$ is $S_{p^- - 1}$,
and these edges meet at $Q_{p^+}^{p^-}$.

\begin{proposition}
Suppose $F_p$ has $\frac{m}{p}$-form.
\begin{enumerate}[label=\Roman*),ref=\Roman*,itemsep=0mm]
\item
If $Q_{p^-}^{p^+}$ is located below $\Sigma_1$, and $Q_{p^+}^{p^-}$ is located above $\Sigma_1$,
then $z_p \in \left( z_{p^-},z_{p^+} \right)$,
$I_p = \left[ z_p, z_{p^+} \right)$,
and $F_{p^- + p}$ has $\frac{m^- + m}{p^- + p}$-form.
\item
If $Q_{p^-}^{p^+}$ is located above $\Sigma_1$, and $Q_{p^+}^{p^-}$ is located below $\Sigma_1$,
then $z_{p^-} \in \left( z_p,z_{p^+} \right)$,
$I_p = \left( z_{p^-}, z_p \right]$,
and $F_{p + p^+}$ has $\frac{m + m^+}{p + p^+}$-form.
\item
Otherwise $Q_{p^-}^{p^+}$ and $Q_{p^+}^{p^-}$ are both located on or below $\Sigma_1$, and
$I_p = \left( z_{p^-}, z_{p^+} \right)$.
\end{enumerate}
\label{pr:threeCases}
\end{proposition}

\begin{figure}[b!]
\begin{center}
\includegraphics[width=15cm]{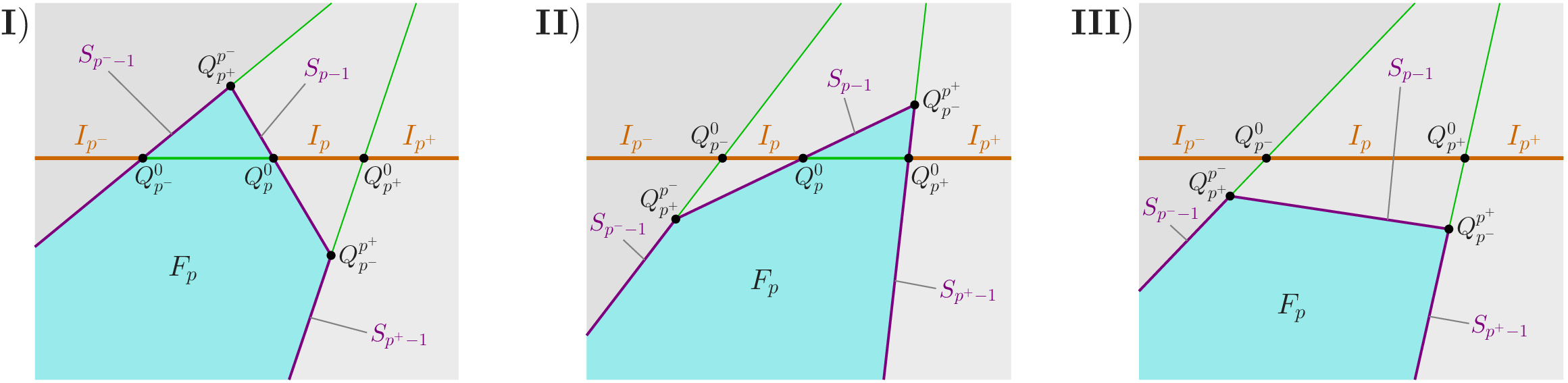}
\caption{
The three cases of Proposition \ref{pr:threeCases}.
\label{fig:divP}
} 
\end{center}
\end{figure}

The three cases of Proposition \ref{pr:threeCases} are illustrated in Fig.~\ref{fig:divP}.
Proposition \ref{pr:threeCases} shows that it is not possible
for $Q_{p^-}^{p^+}$ and $Q_{p^+}^{p^-}$ to both lie above $\Sigma_1$.
Below we will see that the three cases of Proposition \ref{pr:threeCases}
correspond to the three cases in Step 3 of the algorithm in Theorem \ref{th:construct}.

To verify part I of Proposition \ref{pr:threeCases},
it necessary to construct $F_{p^- + p}$ from $F_p$.
This is achieved by characterising $F_{p+k}$, for all $k = 0,1,\ldots,p^-$, Lemma \ref{le:caseI},
which is turn is achieved by performing induction on $k$.
First we characterise $I_p$ in case I:

\begin{lemma}
In case I of Proposition \ref{pr:threeCases},
$z_p \in \left( z_{p^-}, z_{p^+} \right)$,
$I_p = \left[ z_p, z_{p^+} \right)$,
$Q_{p^-}^0 \in U_R^{p-1}$, and $Q_{p^+}^0 \in U_L^{p-1}$.
\label{le:caseIeasyBit}
\end{lemma}

\begin{proof}
The edge $S_{p^-}$ of $F_p$ has vertex $Q_{p^-}^1 \in S_0$,
while the edge $S_{p^+}$ of $F_p$ has vertex $Q_{p^+}^1 \in S_0$.
By \eqref{eq:FrU} and the convexity of $F_p$,
if $P \notin S_j$ belongs to the closure of $F_p$, then $P \in U_R^j$.
Thus $Q_{p^-}^1 \in U_R^{p^+}$ and $Q_{p^+}^1 \in U_R^{p^-}$.

Notice $Q_{p^-}^1$ is situated below $Q_{p^+}^1$ by Lemma \ref{le:verticesFp}.
Thus $U_R^{p^+}$ is the open half-plane consisting of all points below $\Sigma_{-p^+}$,
while $U_R^{p^-}$ is the open half-plane consisting of all points above $\Sigma_{-p^-}$.

By mapping these under $f_R$,
we deduce that $U_R^{p^+ - 1}$ is the open half-plane consisting
of all points to the left of $\Sigma_{-(p^+ - 1)}$,
while $U_R^{p^- - 1}$ is the open half-plane consisting of all points to the right of $\Sigma_{-(p^- - 1)}$.
The lines $\Sigma_{-(p^+ - 1)}$ and $\Sigma_{-(p^- - 1)}$ intersect $\Sigma_1$ at
$Q_{p^+}^0 = (z_{p^+},0)$ and $Q_{p^-}^0 = (z_{p^-},0)$ respectively.
Thus $x \in \left( z_{p^-}, z_{p^+} \right)$ if and only if $(x,0) \in U_R^{p^- - 1} \cap U_R^{p^+ - 1}$.

The point $Q_{p^+}^{p^-}$ belongs to $S_{p^- - 1} \cap S_{p-1}$ but not $S_{p^+ - 1}$,
thus $Q_{p^+}^{p^-} \in U_R^{p^+ - 1}$.
Similarly, $Q_{p^-}^{p^+}$ belongs to $S_{p^+ - 1} \cap S_{p-1}$ but not $S_{p^- - 1}$,
thus $Q_{p^-}^{p^+} \in U_R^{p^- - 1}$.
Thus the edge $S_{p-1}$, minus its endpoints, belongs to $U_R^{p^- - 1} \cap U_R^{p^+ - 1}$.
In case I this edge intersects $\Sigma_1$, necessarily at
$(z_p,0)$, so $z_p$ is well-defined and belongs to $\left( z_{p^-}, z_{p^+} \right)$.

In case I, the vertices of $S_{p^- - 1}$ lie on different sides of $\Sigma_1$.
Thus $Q_{p^-}^0$, where the line through $S_{p^- - 1}$ intersects $\Sigma_1$, belongs to $S_{p^- - 1}$.
So $Q_{p^-}^0 \notin S_{p-1}$ belongs to the closure of $F_p$, hence belongs to $U_R^{p-1}$.
Thus $U_R^{p-1}$ consists of all points to the left of $\Sigma_{-(p-1)}$,
and so $Q_{p^+}^0 \in U_L^{p-1}$.
Finally, by Lemma \ref{le:Jp}, $I_p$ is the set of all $x \in \left( z_{p^-}, z_{p^+} \right)$
for which $(x,0) \notin U_R^{p-1}$, thus $I_p = \left[ z_p, z_{p^+} \right)$.
\end{proof}

\begin{lemma}
Consider case I in Proposition \ref{pr:threeCases}.
For all $k = 0,1,\ldots,p^-$,
\begin{enumerate}[label=\roman*),ref=\roman*,itemsep=0mm]
\item
$F_{p+k}$ is adolescent,
\item
the vertices of $F_{p+k}$ are
\begin{equation}
Q_{p^-}^1, Q_{p^-}^2, \ldots, Q_{p^-}^{p^+ + k}, \qquad
Q_{p^+}^{k+1}, Q_{p^+}^{k+2}, \ldots, Q_{p^+}^{p^-}, \qquad
Q_p^1, Q_p^2, \ldots, Q_p^k,
\end{equation}
\item
$\partial \nu(F_{p+k})$ contains $p^+ + k$ instances of $p^-$,
$p^- - k$ instances of $-p^+$,
and $k$ instances of $p$,
\item
if $k < p^-$, then all vertices of $F_{p+k}$ lie below $\Sigma_1$ except $Q_{p^+}^{p^-}$.
\item
if $k < p^-$, then $Q_{p^-}^0$ belongs to the edge $S_{p^- - 1}$ of $F_{p+k}$,
and $Q_p^0$ belongs to the edge $S_{p - 1}$ of $F_{p+k}$.
\end{enumerate}
\label{le:caseI}
\end{lemma}

\begin{proof}
We prove Lemma \ref{le:caseI} by induction on $k$.
With $k = 0$, part (i) of Lemma \ref{le:caseI} is true because $F_p$ has $\frac{m}{p}$-form,
part (ii) is true by Lemma \ref{le:verticesFp}(ii),
part (iii) is true by Lemma \ref{le:differences},
part (iv) is true by Definition \ref{df:mpForm}(iii) and the assumptions in case I,
and (v) is true because $Q_{p^-}^0 \in S_{p^- - 1}$ and $Q_p^0 \in S_{p-1}$,
as noted in the proof of Lemma \ref{le:caseIeasyBit}.

Thus Lemma \ref{le:caseI} is true for $k = 0$.
Suppose Lemma \ref{le:caseI} is true for some $k \in \{ 0,1,\ldots,p^- - 1 \}$.
This is our induction hypothesis.
For brevity, we write IH(i) for the assumption that part (i) is true for the given value of $k$,
and write $\text{IH(ii)},\ldots,\text{IH(v)}$ analogously.
To complete the proof we verify (i)--(v) for $k+1$.

By IH(ii), the point $Q_{p^+}^{k+1}$ is a vertex of $F_{p+k}$.
By mapping $Q_{p^+}^0 \in U_L^{p-1}$ (Lemma \ref{le:caseIeasyBit})
under $f_R^{-(k+1)}$, we obtain $Q_{p^+}^{k+1} \in U_L^{p+k}$.
We now show every other vertex of $F_{p+k}$ belongs instead to $U_R^{p+k}$.
For any $j$ with $k+2 \le j \le p^-$,
by Lemma \ref{le:verticesFp} the point $Q_{p^+}^{j-k-1}$
is a vertex of $F_p$ not belonging the edge $S_{p-1}$ of $F_p$,
thus $Q_{p^+}^{j-k-1} \in U_R^{p-1}$ by \eqref{eq:FrU},
and by mapping this under $f_R^{-(k+1)}$ we obtain $Q_{p^+}^j \in U_R^{p+k}$.
For any $j$ with $1 \le j \le p^+ + k$,
by IH(ii) and IH(v) the point $Q_{p^-}^{j-1}$ belongs
to the boundary of $F_{p+k}$ but not the edge $S_{p+k-1}$ of $F_{p+k}$,
thus $Q_{p^-}^{j-1} \in U_R^{p+k-1}$ by \eqref{eq:FrU},
and by mapping this under $f_R^{-1}$ we obtain $Q_{p^-}^j \in U_R^{p+k}$.
Similarly for any $j$ with $1 \le j \le k$,
by IH(ii) and IH(v) the point $Q_p^{j-1}$ belongs
to the boundary of $F_{p+k}$ but not the edge $S_{p+k-1}$ of $F_{p+k}$,
thus $Q_p^{j-1} \in U_R^{p+k-1}$ by \eqref{eq:FrU},
and by mapping this under $f_R^{-1}$ we obtain $Q_p^j \in U_R^{p+k}$.

By \eqref{eq:FrU}, $F_{p+k+1}$ is formed by cutting $F_{p+k}$ along $\Sigma_{-(p+k)}$
and retaining the part in $U_R^{p+k}$.
Since $Q_{p^+}^{k+1}$ belongs to $U_L^{p+k}$,
while every other vertex of $F_{p+k}$ belongs to $U_R^{p+k}$,
the only edges of $F_{p+k}$ that meet $\Sigma_{-(p+k)}$
are those adjacent to the `cut vertex' $Q_{p^+}^{k+1}$.
By \eqref{eq:intersectionFormula} and IH(iii),
the edge of $F_{p+k}$ clockwise from $Q_{p^+}^{k+1}$ is $S_{p^+ + k}$,
and the edge of $F_{p+k}$ anticlockwise from $Q_{p^+}^{k+1}$ is $S_k$.
The new edge (of $F_{p+k+1}$) is a subset of $\Sigma_{-(p+k)}$,
so meets $S_{p^+ + k}$ at $Q_{p^-}^{p^+ + k + 1}$ and $S_k$ at $Q_p^{k + 1}$.

So in going from $F_{p+k}$ to $F_{p+k+1}$,
we gain an edge along $\Sigma_{-(p+k)}$ verifying (i).
Also we lose the vertex $Q_{p^+}^{k+1}$
and gain vertices $Q_{p^-}^{p^+ + k + 1}$ and $Q_p^{k + 1}$, verifying (ii).
The difference vector $\partial \nu(F_{p+k+1})$ is the same as $\partial \nu(F_{p+k})$
but has one less $-p^+$,
an additional $p^-$,
and an additional $-p$, verifying (iii).

For the remainder of proof we assume $k < p^- - 1$.
By IH(iv), the only edges of $F_{p+k}$ that contain points above $\Sigma_1$
are $S_{p-1}$ and $S_{p^- - 1}$ adjacent to $Q_{p^+}^{p^-}$.
If the cut vertex is not an endpoint of $S_{p-1}$ or $S_{p^- - 1}$,
then the new vertices lie below $\Sigma_1$ verifying (iv),
and (v) is an immediate consequence of IH(v) because $Q_{p^-}^0, Q_p^0 \in \Sigma_1$.

The cut vertex is $Q_{p^+}^{k+1}$ and the endpoints of $S_{p-1}$ are $Q_{p^+}^{p^-}$ and $Q_{p^-}^{p^+}$,
so the cut vertex is not an endpoint of $S_{p-1}$ because $k \ne p^- - 1$ and $p^+ \ne p^-$.
The endpoints of $S_{p^- - 1}$ are $Q_{p^+}^{p^-}$
and either $Q_{p^-}^{p^-}$ (occurring if $p^- < p^+$)
or $Q_{p^+}^{p^- - p^+}$ (occurring if $p^- > p^+$).
Of these the cut vertex can only be $Q_{p^+}^{p^- - p^+}$ when $p^- > p^+$ and $k = p^- - p^+ - 1$.
It remains to show that in this case the new vertex
$Q_{p^-}^{p^-} \in S_{p^- - 1}$ lies below $\Sigma_1$ for this will verify (iv) and (v).

Given $P, Q \in \mathbb{R}^2$,
we say $Q$ is {\em anticlockwise from $P$ relative to $X$}
if $P$, $Q$, and $X$ are the vertices of a triangle,
and as we go from $P$ to $Q$ we travel anticlockwise around one edge of the triangle.
The point $Q_{p^+}^{p^-}$ is anticlockwise from $Q_p^0$ relative to $X$,
due to the position of these points on $S_{p-1}$.
Also $z_p < z_{p^+}$ (Lemma \ref{le:caseIeasyBit}),
thus $Q_{p^+}^{p^-}$ is anticlockwise from $Q_{p^+}^0$ relative to $X$.
But $Q_{p^+}^0 = f_R^{p^-} \left( Q_{p^+}^{p^-} \right)$,
and $f_R^{p^-}$ is a linear rotation centred about $X$
(because $\delta_R > \frac{\tau_R^2}{4}$, see Definition \ref{df:mpForm}).
Thus {\em every} $P \ne X$ is anticlockwise from $f_R^{p^-}(P)$ relative to $X$.
In particular, $Q_{p^-}^{p^-}$ is anticlockwise from
$Q_{p^-}^0 = f_R^{p^-} \left( Q_{p^-}^{p^-} \right) \in \Sigma_1$ relative to $X$,
thus $Q_{p^-}^{p^-}$ lies below $\Sigma_1$.
\end{proof}

\begin{proof}[Proof of Proposition \ref{pr:threeCases}]
The polygon $F_p$ is the interior of the convex hull of its vertices.
By Lemma \ref{le:verticesFp}(ii), these vertices are
$Q_{p^-}^1, Q_{p^-}^2, \ldots, Q_{p^-}^{p^+}$ and $Q_{p^+}^1, Q_{p^+}^2, \ldots, Q_{p^+}^{p^-}$.
Since $f_R$ is affine, $f_R(F_p)$ is the interior
of the convex hull of the images of the these points under $f_R$:
$Q_{p^-}^0, Q_{p^-}^1, \ldots, Q_{p^-}^{p^+ - 1}$ and $Q_{p^+}^0, Q_{p^+}^1, \ldots, Q_{p^+}^{p^- - 1}$.

Suppose for a contradiction that $Q_{p^-}^{p^+}$ and $Q_{p^+}^{p^-}$ both lie on or above $\Sigma_1$.
In this case the entire edge $S_{p-1}$ of $F_p$ lies on or above $\Sigma_1$,
while all other vertices of $F_p$ lies below $\Sigma_1$ by Definition \ref{df:mpForm}(iii).
The edge clockwise from $S_{p-1}$ meets $\Sigma_1$ at $Q_{p^+}^0$,
while the edge anticlockwise from $S_{p-1}$ meets $\Sigma_1$ at $Q_{p^-}^0$.
Thus the part of $F_p$ that lies below $\Sigma_1$ is the interior of the convex hull of
$Q_{p^-}^0, Q_{p^-}^1, \ldots, Q_{p^-}^{p^+ - 1}$ and $Q_{p^+}^0, Q_{p^+}^1, \ldots, Q_{p^+}^{p^- - 1}$.
That is, $f_R(F_p) \subset F_p$.
This is a contradication because $\delta_R > 1$ (see Definition \ref{df:mpForm})
so $f_R$ is area-expanding.

Now suppose $Q_{p^-}^{p^+}$ or $Q_{p^+}^{p^-}$ lie on or below $\Sigma_1$ (case III).
In this case $F_p$ lies entirely below $\Sigma_1$.
For any $j \ge p$, $E_j \subset F_p$, by \eqref{eq:Fr},
thus $E_j \cap \Sigma_1 = \varnothing$,
and hence $I_{j+1} = \varnothing$ by \eqref{eq:In2}.
So $J_p = I_p$ by \eqref{eq:J},
and $I_p = \left( z_{p^-}, z_{p^+} \right)$ by Lemma \ref{le:Jp}.

Finally suppose $Q_{p^-}^{p^+}$ is located below $\Sigma_1$,
and $Q_{p^+}^{p^-}$ is located above $\Sigma_1$ (case I).
(We omit a verification for case II as this can be proved in the same fashion as case I.)
In view of Lemma \ref{le:caseIeasyBit}, it remains to show
$F_{p^- + p}$ has $\frac{m^- + m}{p^- + p}$-form.
To do this we use Lemma \ref{le:caseI}
to show that $F_{p^- + p}$ satisfies all three properties of Definition \ref{df:mpForm}.

By Lemma \ref{le:caseI}(i), $F_{p^- + p}$ is adolescent.
By Lemma \ref{le:caseI} (iii),
$\partial \nu \left( F_{p^- + p} \right)$ contains $p^+ + p^- = p$ instances of $p^-$,
and $p^-$ instances of $p$.
This permutation also starts with $0$ by definition,
thus $\nu \left( F_{p^- + p} \right) = \cP_{m^- + m, p^- + p}$ by Lemma \ref{le:differences}.
The upper vertex of the edge $S_0$ of $F_{p^- + p}$ lies on or below $\Sigma_1$
because this property holds for $F_p$, which contains $F_{p^- + p}$.
Finally, by Lemma \ref{le:caseI}(iv) with $k = p^- - 1$,
all vertices of $F_{p^- + p - 1}$ lie below $\Sigma_1$ except $Q_{p^+}^{p^-}$
(which is not a vertex of $F_{p^- + p}$).
Thus all vertices of $F_{p^- + p}$ lie below $\Sigma_1$,
except possibly those belonging to the edge $S_{p^- + p -1}$.
\end{proof}

\subsection{Final arguments}
\label{sub:final}

\begin{definition}
An irreducible fraction $\frac{m}{p} \in \left( 0, \frac{1}{n^*} \right)$
is a {\em constituent} of $f_R$ if
\begin{enumerate}[label=\roman*),ref=\roman*,itemsep=0mm]
\item
$F_p$ has $\frac{m}{p}$-form,
\item
for any $n < p$, $I_n \ne \varnothing$ if and only if $n \in D[m,p;n^*]$,
\item
$J_p$ and the non-empty $I_n$ with $n < p$ are ordered by $D[m,p;n^*]$.
\end{enumerate}
\label{df:atom}
\end{definition}

\begin{lemma}
Suppose $\frac{m}{p}$ is a constituent of $f_R$.
\begin{enumerate}[label=\Roman*),ref=\Roman*,itemsep=0mm]
\item
If $Q_{p^-}^{p^+}$ is located below $\Sigma_1$, and $Q_{p^+}^{p^-}$ is located above $\Sigma_1$,
then $\frac{m^- + m}{p^- + p}$ is a constituent of $f_R$.
\item
If $Q_{p^-}^{p^+}$ is located above $\Sigma_1$, and $Q_{p^+}^{p^-}$ is located below $\Sigma_1$,
then $\frac{m + m^+}{p + p^+}$ is a constituent of $f_R$.
\end{enumerate}
\label{le:inductiveStep}
\end{lemma}

\begin{proof}
By Definition \ref{df:atom}, $F_p$ has $\frac{m}{p}$-form,
so $J_p = \left( z_{p^-}, z_{p^+} \right)$ by Lemma \ref{le:Jp}.
Consider case I that $Q_{p^-}^{p^+}$ is located below $\Sigma_1$,
and $Q_{p^+}^{p^-}$ is located above $\Sigma_1$
(case II can be proved in the same fashion).
It remains to show $\frac{m^- + m}{p^- + p}$ obeys Definition \ref{df:atom}.

By Proposition \ref{pr:threeCases}(I),
$F_{p^- + p}$ has $\frac{m^- + m}{p^- + p}$-form,
verifying condition (i) of Definition \ref{df:atom}.
Also $I_p = \left[ z_p, z_{p^+} \right)$, while Lemma \ref{le:Jp} implies
$J_p = \left( z_{p^-}, z_{p^+} \right)$ and $J_{p^- + p} = \left( z_{p^-}, z_p \right)$.

Since $\frac{m}{p}$ is a constituent,
$J_p$ and the nonempty $I_n$ with $n < p$ are ordered by $D[m,p;n^*]$.
By replacing $J_p$ with $J_{p^- + p}$ and $I_p$ (in that order),
we obtain $J_{p^- + p}$ and the nonempty $I_n$ with $n < p^- + p$.
These intervals are ordered by $D[m^- + m,p^- + p;n^*]$
because by Definition \ref{df:denominatorList}
this list can be formed by taking $D[m,p;n^*]$
and inserting $p^- + p$ immediately before $p$.
This verifies conditions (ii) and (iii) of Definition \ref{df:atom} for $\frac{m^- + m}{p^- + p}$.
\end{proof}

\begin{proof}[Proof of Theorem \ref{th:existence}]
We first show that $\frac{1}{n^* + 1}$ is a constituent of $f_R$.
Recall $I_1 = (-\infty,0]$, while by Proposition \ref{pr:nStar},
$I_n = \varnothing$ for all $n = 2,3,\ldots, n^* - 1$, and
$I_{n^*} = \left[ z_{n^*}, 1 \right]$, where $0 < z_{n^*} \le 1$.

By Lemma \ref{le:baseCase}, $F_{n^*+1}$ has $\frac{1}{n^* + 1}$-form. 
Then by Lemma \ref{le:Jp}, $J_{n^*+1} = \left( z_1, z_{n^*} \right)$, where $z_1 = 0$.
By Definition \ref{df:denominatorList}, $D[1,n^*+1;n^*] = (1,n^*+1,n^*)$,
thus $\frac{1}{n^* + 1}$ is a constituent of $f_R$ by definition.

We now apply Lemma \ref{le:inductiveStep} recursively
starting with $\frac{m}{p} = \frac{1}{n^* + 1}$ to generate a sequence of constituents.
Each constituent belongs to $\left[ \frac{1}{n^*+1}, \frac{1}{n^*} \right)$
because case I cannot occur in our first application of Lemma \ref{le:inductiveStep}
since $Q_{n^*}^1 = \left( 0, z_{n^*} - 1 \right)$ lies on or below $\Sigma_1$.
The sequence of constituents is finite for otherwise $I_n \ne \varnothing$
for arbitrarily large values of $n$, which is not possible because $f_R$ is a repelling focus.
That is, the sequence terminates at a constituent
$\frac{m}{p} \in \left[ \frac{1}{n^*+1}, \frac{1}{n^*} \right)$
that does not satisfy cases I or II of Lemma \ref{le:inductiveStep}.

This value satisfies case III of Proposition \ref{pr:threeCases}.
Thus $I_p = J_p$, using also Lemma \ref{le:Jp},
so $I_n = \varnothing$ for all $n > p$.
Then by property (ii) of a constituent,
$I_n = \varnothing$ if and only if $n \in D[m,p;n^*]$ for all $n \ge 1$.
By property (iii) of a constituent,
the nonempty $I_n$ are ordered by $D[m,p;n^*]$.
The value $\frac{m}{p}$ in Theorem \ref{th:existence} is unique
because any other irreducible fraction yields a different denominator list.
\end{proof}

\begin{proof}[Proof of Theorem \ref{th:construct}]
Step 1 of Theorem \ref{th:construct} generates
$I_1 = (-\infty,0]$ and $I_{n^*} = \left[ z_{n^*}, 1 \right]$,
in accordance with Proposition \ref{pr:nStar}.
Steps 2 and 3 generate a sequence of irreducible fractions
starting with $\frac{m}{p} = \frac{1}{n^* + 1}$.

In each instance of Step 3,
$z^- = z_{p^-}$ and $z^+ = z_{p^+}$.
If $\alpha_{p-1} < 0$ and $z_p \in (z^-, z^+)$,
then by \eqref{eq:fRr1} $f_R^{p-1}(z_p + \ee, 0) \in \Omega_L$ and
$f_R^{p-1}(z_p - \ee, 0) \in \Omega_R$ for $\ee > 0$.
Thus $z_p + \ee \in I_p$ and $z_p - \ee \notin I_p$ for sufficiently small $\ee > 0$,
so case I of Proposition \ref{pr:threeCases} occurs.
Step 3 generates $I_p = [z_p,z^+)$, as in Proposition \ref{pr:threeCases},
and the next irreducible fraction is $\frac{m^- + m}{p^- + p}$,
listed in case I of Lemma \ref{le:inductiveStep}.

If $\alpha_{p-1} > 0$ and $z_p \in (z^-, z^+)$,
then case II of Proposition \ref{pr:threeCases} occurs.
Step 3 generates $I_p = (z^-, z_p]$, as in Proposition \ref{pr:threeCases},
and increments to $\frac{m + m^+}{p + p^+}$,
listed in case II of Lemma \ref{le:inductiveStep}.

Otherwise $\alpha_{p-1} = 0$ or $z_p \notin (z^-, z^+)$,
corresponding to case III of Proposition \ref{pr:threeCases}.
Step 3 generates $I_p = (z^-,z^+)$,
as in Proposition \ref{pr:threeCases},
then outputs $\frac{m}{p}$ and terminates the algorithm.

In conclusion, the sequence of irreducible fractions
generated by the algorithm in Theorem \ref{th:construct}
is identical to the sequence of constituents constructed
in the proof of Theorem \ref{th:existence}.
Therefore, the fraction $\frac{m}{p}$ generated by the algorithm
is the same as that of Theorem \ref{th:existence}.
Moreover, in Step 1 and in each instance of Step 3
the algorithm correctly generates all nonempty $I_n$
by Propositions \ref{pr:nStar} and \ref{pr:threeCases}. 
\end{proof}

\section{Additional results}
\label{sec:more}

\subsection{Minimal polygons}
\label{sub:minimalPolygons}

By definition, $F_0 = \mathbb{R}^2$, and
\begin{equation}
F_r = \left\{ P \in \mathbb{R}^2 \,\middle|\, f_R^j(P) \in \Omega_R \text{~for all $j = 0,1,\ldots,r-1$} \right\},
\label{eq:Fr5}
\end{equation}
for all $r \ge 1$.
If $\delta_R \ne 0$, then we can instead use backwards iterates, defining $G_0 = \mathbb{R}^2$ and
\begin{equation}
G_r = \left\{ P \in \mathbb{R}^2 \,\middle|\, f_R^{-j}(P) \in \Omega_R \text{~for all $j = 0,1,\ldots,r-1$} \right\},
\label{eq:Gr5}
\end{equation}
for all $r \ge 1$.
These sets are bounded by the images $\Sigma_r = f_R^r(\Sigma_0)$ of $\Sigma_0$ under $f_R$.

Now suppose $\delta_R > 0$ and $\delta_R > \frac{\tau_R^2}{4}$, so that $X$ is a repelling focus belonging to $\Omega_R$.
For all $r \ge 0$,
if $\Sigma_r$ intersects $G_r$,
then $G_{r+1}$ is the result of cutting $G_r$ along $\Sigma_r$
and retaining the component that contains $X$.
The lines $\Sigma_r$ grow more distant from $X$ as $r$ increases,
so eventually no more cuts are taken,
and all subsequent $G_r$ are identical.
That is, the $G_r$ converge in finitely many iterations to the set
\begin{equation}
G = \left\{ P \in \mathbb{R}^2 \,\middle|\, f_R^{-j}(P) \in \Omega_R \text{~for all $j \ge 0$} \right\},
\label{eq:G}
\end{equation}
which is necessarily a polygon.

Szalai and Osinga \cite{SzOs08} refer to $G$ as the {\em minimal polygon}.
They use the assumption that its
edges are ordered by a rotation permutation $\cP_{m,p}$
to determine which edges may vanish under parameter variation.
Here we prove this characterisation of the edges.

\begin{proposition}
Suppose $\delta_R > 1$ and $\delta_R > \frac{\tau_R^2}{4}$,
and let $\frac{m}{p}$ be as in Theorem \ref{th:existence}.
Then $G$ is a $p$-sided polygon
whose edges are subsets of $\Sigma_0, \Sigma_1, \ldots, \Sigma_{p-1}$ ordered clockwise by $\cP_{m,p}$.
Moreover,
\begin{equation}
G = f_R^{p-1}(F_p).
\label{eq:G6}
\end{equation}
\label{pr:Gr}
\end{proposition}

Fig.~\ref{fig:divU} shows $G$ and $F_p$ for parameter point {\bf c},
for which $\frac{m}{p} = \frac{2}{9}$.
The edges of $G$ are labelled $T_0,T_1,\ldots,T_8$, where $T_j \subset \Sigma_j$ for each $j$.
The edges are ordered clockwise by $\cP_{2,9} = (0,5,1,6,2,7,3,8,4)$
in accordance with Proposition \ref{pr:Gr}.

\begin{figure}[b!]
\begin{center}
\includegraphics[width=9cm]{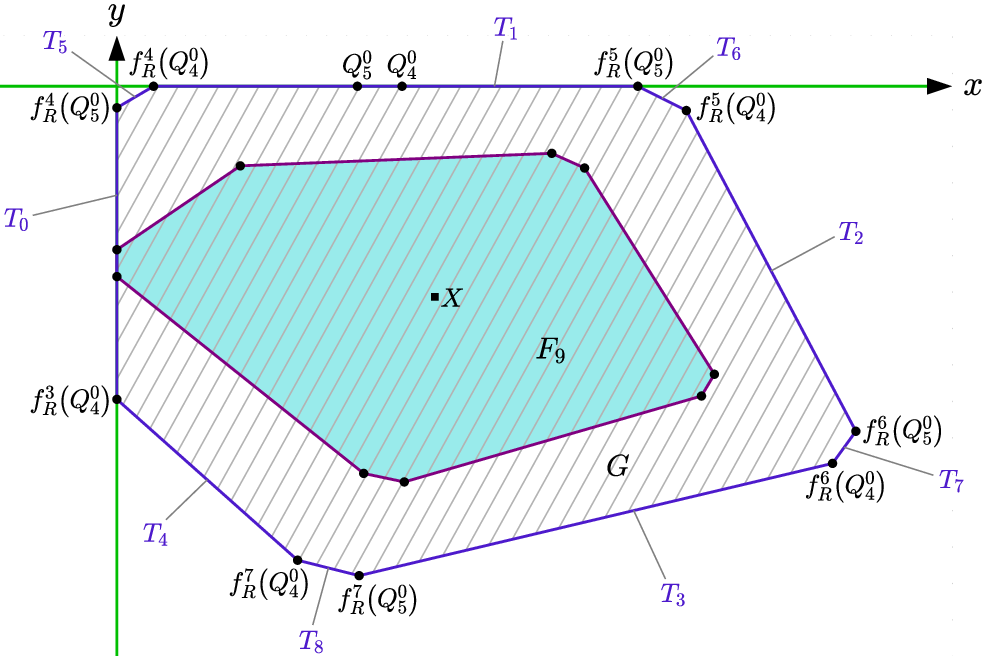}
\caption{
The minimal polygon $G$ (striped)
and $F_9 = f_R^{-8}(G)$ (shaded)
for parameter point {\bf c}.
The edges of $G$ are $T_j \subset \Sigma_j$ for $j = 0,1,\ldots,8$.
The vertices of $G$ belong to the forward orbits of $Q_4^0$ and $Q_5^0$ under $f_R$.
\label{fig:divU}
} 
\end{center}
\end{figure}

To prove Proposition \ref{pr:Gr} we first establish the following result.

\begin{lemma}
Suppose $\delta_R > 1$ and $\delta_R > \frac{\tau_R^2}{4}$,
and let $\frac{m}{p}$ be as in Theorem \ref{th:existence}.
The set $F_p$ has $\frac{m}{p}$-form and lies below $\Sigma_1$.
Moreover, $F_{p+k} = f_R^{-k}(F_p)$ for all $k \ge 0$.
\label{le:adultFp}
\end{lemma}

\begin{proof}
In the proof of Theorem \ref{th:existence},
$\frac{m}{p}$ was constructed as a constituent conforming to
case III of Proposition \ref{pr:threeCases}.
Thus $F_p$ has $\frac{m}{p}$-form, by Definition \ref{df:atom},
with $S_{p-1}$ situated on or below $\Sigma_1$.
Thus $F_p$ lies below $\Sigma_1$ by Definition \ref{df:mpForm}.
Then by \eqref{eq:Fr4}, $F_{p+1} = f_R^{-1}(F_p)$,
which lies below $\Sigma_1$ because $F_{p+1} \subset F_p$ by \eqref{eq:Fr}.
Thus we can apply \eqref{eq:Fr4} repeatedly giving
$F_{p+k} = f_R^{-k}(F_p)$ for all $k \ge 0$.
\end{proof}

\begin{proof}[Proof of Proposition \ref{pr:Gr}]
By \eqref{eq:Fr5} and \eqref{eq:Gr5},
\begin{equation}
G_r = f_R^{r-1}(F_r), \qquad \text{for all $r \ge 1$}.
\label{eq:Gr6}
\end{equation}
Choose any $k \ge 0$ and substitute $r = p+k$ into \eqref{eq:Gr6}.
Since $F_{p+k} = f_R^{-k}(F_p)$ (Lemma \ref{le:adultFp}),
this gives $G_{p+k} = f_R^{p-1}(F_p)$.
This identity holds for all $k \ge 0$, verifying \eqref{eq:G6}.

Since $F_p$ has $\frac{m}{p}$-form (Lemma \ref{le:adultFp}),
it is a $p$-sided polygon whose edges $S_j \in \Sigma_{-j}$ are 
ordered anticlockwise by $\cP_{m,p}$.
By \eqref{eq:G6}, $G$ has edges $f_R^{p-1}(S_j)$, ordered anticlockwise by $\cP_{m,p}$.
For each $0 \le j \le p-1$, let $T_j = f_R^{p-1}(S_{p-j-1})$, so that $T_j \subset \Sigma_j$.
Since $j \mapsto p-j-1$ reverses the order of the indexing, the $T_j$ are ordered clockwise by $\cP_{m,p}$.
\end{proof}

\subsection{Parameter region boundaries}
\label{sub:polarCoords}

Consider the map \eqref{eq:f} with $\delta_R > 0$ and $\delta_R > \frac{\tau_R^2}{4}$,
and let $\frac{m}{p}$ be as in Theorem \ref{th:existence}.
As we alter the values of $\tau_R$ and $\delta_R$,
a non-empty interval $I_n$ may shrink to a point and vanish.
However, by Theorem \ref{th:existence} this is only possible for $n = p$,
because if we remove an interval $I_n$ with $n < p$,
there is no denominator list that gives the ordering of the remaining intervals.

Upon removing $I_p$,
the remaining intervals are ordered by $D[m^-,p^-;n^*]$ if $p^- > p^+$,
and ordered by $D[m^+,p^+;n^*]$ if $p^- < p^+$.
That is, the rotation number associated with $f$
changes from $\frac{m}{p}$ to its parent with the larger denominator,
i.e.~its {\em younger} parent \cite{LaTr95}. 

This behaviour is evident in Fig.~\ref{fig:numPiecesWithRotNums}.
From any parameter region corresponding to an irreducible fraction $\frac{m}{p}$,
as we cross the upper boundary of this region we enter the region corresponding to the younger parent.
We now derive a formula for this boundary.

The interval $I_p$ has endpoints $z_{p^-}$ and $z_{p^+}$,
thus vanishes when $z_{p^-} = z_{p^+}$.
By \eqref{eq:zn}, this equation is equivalent to
\begin{equation}
\alpha_{p^+ - 1} \beta_{p^- - 1} = \alpha_{p^- - 1} \beta_{p^+ - 1} \,.
\nonumber
\end{equation}
By then substituting \eqref{eq:alphabeta} we obtain
\begin{align}
&\delta_R^{\frac{p^+ - 1}{2}} \sin \left( p^+ \phi \right) \sin(\phi)
+ \delta_R^{\frac{p^+ + p^- - 1}{2}} \sin \left( p^+ \phi \right) \sin \left( \left( p^- - 1 \right) \phi \right) \nonumber \\
&= \delta_R^{\frac{p^- - 1}{2}} \sin \left( p^- \phi \right) \sin(\phi)
+ \delta_R^{\frac{p^+ + p^- - 1}{2}} \sin \left( p^- \phi \right) \sin \left( \left( p^+ - 1 \right) \phi \right), \nonumber
\end{align}
and by using standard trigonometric identities this reduces to
\begin{equation}
\delta_R^{\frac{p^+}{2}} \sin \left( p^+ \phi \right) - \delta_R^{\frac{p^-}{2}} \sin \left( p^- \phi \right) =
\delta_R^{\frac{p}{2}} \sin \left( \left( p^+ - p^- \right) \phi \right).
\label{eq:formula}
\end{equation}

\begin{figure}[b!]
\begin{center}
\includegraphics[width=15cm]{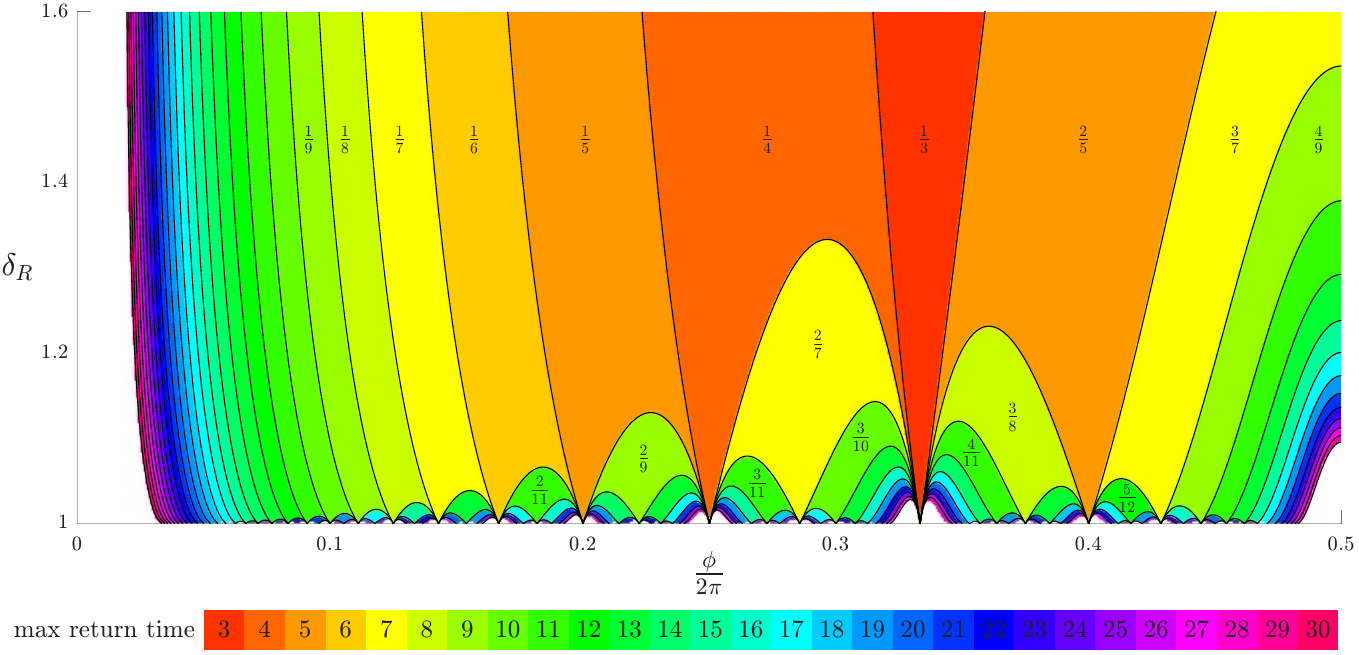}
\caption{
Regions of constant rotation number $\frac{m}{p}$, as defined by Theorem \ref{th:existence},
where $\phi$ is defined by \eqref{eq:phi}.
The value $p$ is the maximum time for orbits of \eqref{eq:f} to return to $\Sigma_1$.
As with Fig.~\ref{fig:numPieces}, these regions were computed numerically
using the algorithm in Theorem \ref{th:construct}.
We also show (in black) all curves \eqref{eq:formula} with $p \le 30$.
\label{fig:numPiecesPolarCoords}
} 
\end{center}
\end{figure}

Fig.~\ref{fig:numPiecesPolarCoords} reproduces Fig.~\ref{fig:numPiecesWithRotNums}
in $\left( \frac{\phi}{2 \pi},\delta_R \right)$ coordinates, and overlays curves
defined by \eqref{eq:formula}.
As expected, the curves match the boundaries of the regions
obtained numerically via Theorem \ref{th:construct}.

Every irreducible $\frac{m}{p} \in \left( 0, \frac{1}{2} \right)$
is the younger parent of exactly two numbers,
$\frac{m^- + m}{p^- + p}$ and $\frac{m + m^+}{p + p^+}$.
Thus each $\frac{m}{p}$-region has two lower boundaries
where it borders the $\frac{m^- + m}{p^- + p}$ and $\frac{m + m^+}{p + p^+}$-regions.
We observe in Fig.~\ref{fig:numPiecesPolarCoords}
that for every $\frac{m}{p} \in \left( 0, \frac{1}{2} \right)$
these two lower boundaries meet at $\left( \frac{\phi}{2 \pi},\delta_R \right) = \left( \frac{m}{p}, 1 \right)$.
Here the rotation number $\frac{m}{p}$ associated with $f$
is the same as the rotation number $\rho_{{\rm fp}} = \frac{\phi}{2 \pi}$ associated with the fixed point $X$.

\subsection{Bifurcations}
\label{sub:bifs}

Fig.~\ref{fig:bifSetWithBoundaries} shows again the two-parameter bifurcation diagram Fig.~\ref{fig:bifSet},
but now overlays (in black) the boundaries of the $\frac{m}{p}$-regions.
These boundaries were produced by
numerically converting the curves of Fig.~\ref{fig:numPiecesPolarCoords} to $(\tau_R,\delta_R)$-coordinates.
We observe that the bifurcation structure is related to the region boundaries.
Many of the shrinking points (or pinch points)
of the periodicity regions lie on region boundaries \cite{SzOs09}.
Some of the region boundaries
are bifurcation curves where the attractor changes from chaotic to non-chaotic
(periodic or quasiperiodic),
such as the vertical line $\tau_R = -1$
that bounds the $\frac{1}{3}$ and $\frac{1}{4}$ regions.

\begin{figure}[b!]
\begin{center}
\includegraphics[width=15cm]{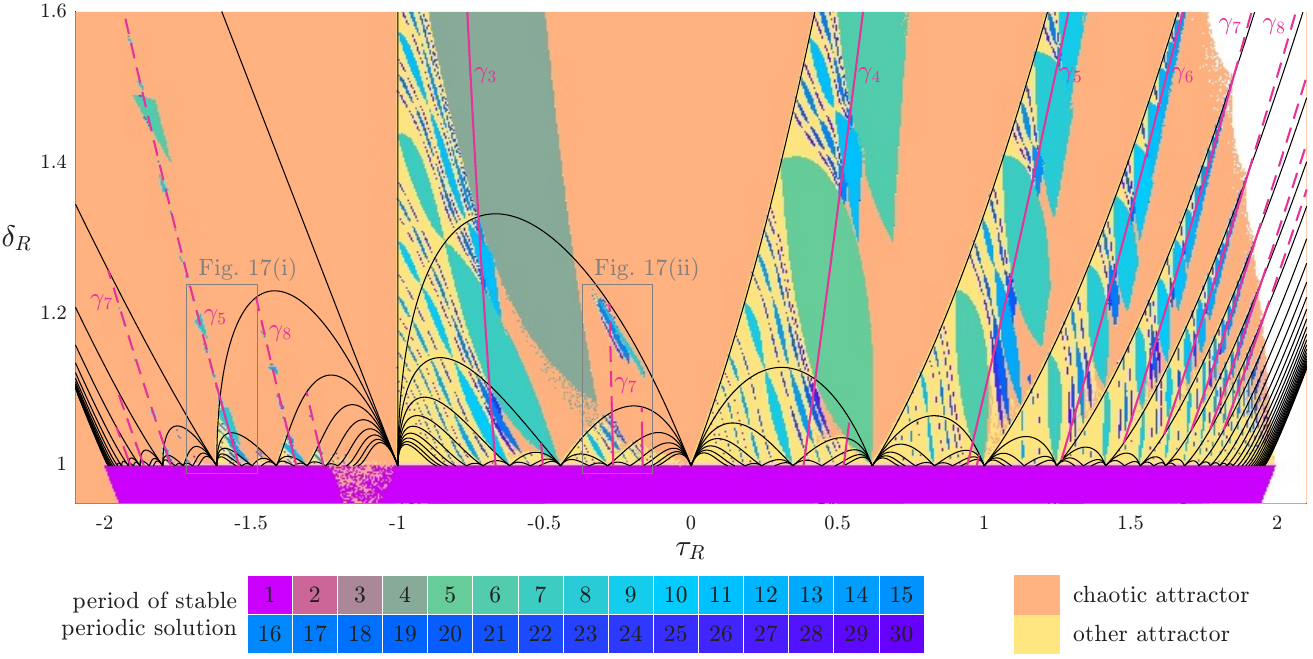}
\caption{
A repeat of Fig.~\ref{fig:bifSet}
but now showing the boundaries (black)
of the regions where $g$ has fixed maximum return time,
and curves $\gamma_n$ (pink) where the slope \eqref{eq:dgdx} vanishes
for the indicated value of $n$.
\label{fig:bifSetWithBoundaries}
} 
\end{center}
\end{figure}

In Fig.~\ref{fig:bifSetWithBoundaries} we have also overlaid (in pink)
some curves where one piece of the first return map $g$ has zero slope.
By Lemma \ref{le:rEqualsNminus1},
each piece of $g$ has the form
$g(x) = f_L \left( f_R^{n-1}(x,0) \right)_1$,
where $n \ge 1$ is the return time.
By \eqref{eq:fRr1}, the first component of $f_R^{n-1}(x,0)$ is $\alpha_{n-1} x + \beta_{n-1}$,
while, from the formula \eqref{eq:fLfR} for $f_R$, the second component of
$f_R^{n-1}(x,0)$ is $\delta_R \left( \alpha_{n-2} x + \beta_{n-2} \right)$, assuming $n \ge 2$ (if $n=1$ then this component is zero).
Thus, by the formula \eqref{eq:fLfR} for $f_L$,
\begin{equation}
g(x) = \tau_L \left( \alpha_{n-1} x + \beta_{n-1} \right) - \delta_R \left( \alpha_{n-2} x + \beta_{n-2} \right) + 1,
\label{eq:g3}
\end{equation}
which has slope
\begin{equation}
g'(x) = \tau_L \alpha_{n-1} - \delta_R \alpha_{n-2},
\label{eq:dgdx}
\end{equation}
assuming $n \ge 2$ (if $n=1$ then $g'(x) = \tau_L$).

Fig.~\ref{fig:bifSetWithBoundaries} shows curves where \eqref{eq:dgdx} vanishes for
viable return times $n \le 11$, and are labelled $\gamma_n$.
Each $\gamma_n$ is drawn dashed in areas where $n$ is the maximum return time, and solid otherwise.
We observe that near some dashed curves
there are collections of periodicity regions that
do not display the usual sausage-string geometry.
These are shown more clearly in the magnified plots of Fig.~\ref{fig:bifSetWithBoundariesZoom}.

\begin{figure}[h!]
\begin{center}
\includegraphics[width=15cm]{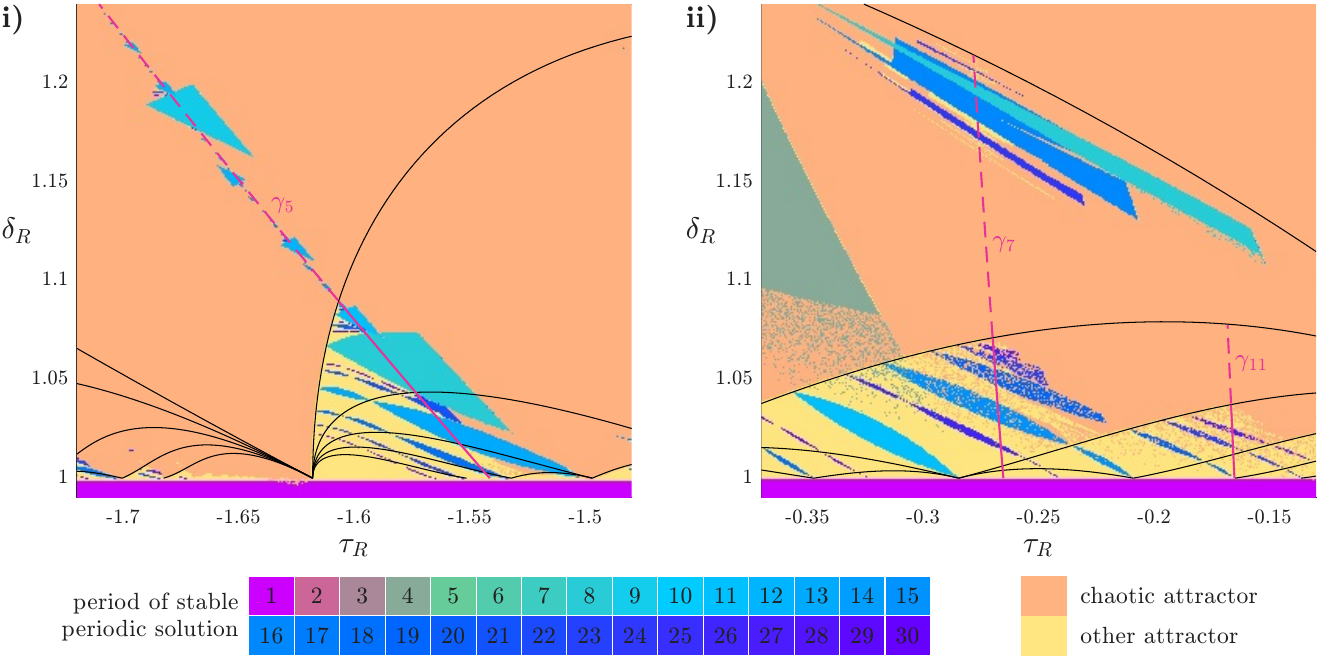}
\caption{
Magnifications of Fig.~\ref{fig:bifSetWithBoundaries}.
\label{fig:bifSetWithBoundariesZoom}
} 
\end{center}
\end{figure}

In the top left area of Fig.~\ref{fig:bifSetWithBoundariesZoom}(i),
the attractor is primarily chaotic.
The exception is around the dashed part of $\gamma_5$
where there is a sequence of roughly triangle-shaped periodicity regions.
We argue that this occurs because the Lyapunov exponent of an orbit of $g$
is the average value of $\ln|g'|$,
so the presence of a piece of $g$ with zero or near-zero slope
inhibits the occurrence of a positive Lyapunov exponent, and thus a chaotic attractor.

In Fig.~\ref{fig:bifSetWithBoundariesZoom}(ii),
the dashed part of $\gamma_7$ is also associated with a novel
collection of periodicity regions.
However, the attractor appears to be robustly chaotic
in the middle part of the plot despite the presence of the curve $\gamma_7$.
Numerical explorations suggest that this occurs because here the attractor
does not visit the piece of $g$ with zero slope.

\section{Discussion}
\label{sec:conc}

Chaotic piecewise-linear maps are multifarious. 
They model neural dynamics \cite{BrSt91,IbCa11},
financial markets \cite{TrWe16,GaRa22b},
and power converters \cite{BaRa00,RoRo02,AvZh20}.
They are employed in encryption algorithms \cite{BeAk07,LiWa11},
characterise corner collisions of piecewise-smooth ODEs \cite{DiBu01c,Co08b},
and used to gain insights into theoretical aspects of chaos \cite{CoLe84,Gl01,SaTa21}.

This paper concerns a prototypical family of piecewise-linear maps.
This family exhibits chaos robustly, but its bifurcation structures remain to be fully understood.
We believe that a deeper understanding of these structures
is attainable via the first return map $g$.
Theorem \ref{th:existence} characterises the configuration of $g$ via a rotation number,
while the algorithm in Theorem \ref{th:construct}
produces the viable return times $n$ and intervals $I_n$ on which $g$ is affine
for a given member of the family.

As parameters are varied, there are three ways in which the configuration of $g$
can change in a codimension-one fashion.
If the rotation number associated with $g$ is $\frac{m}{p}$,
then $I_p$ may vanish, or a new interval $I_{p^- + p}$ or $I_{p + p^+}$ may appear.
This was described in \S\ref{sub:bifs} and is a consequence of Theorem \ref{th:existence}.

As seen in Fig.~\ref{fig:numPiecesPolarCoords},
many regions of fixed rotation number limit to the line $\delta_R = 1$ at exactly three points.
We conjecture that if $\frac{m}{p} \in \left( 0, \frac{1}{2} \right)$ does not
have $\frac{0}{1}$ or $\frac{1}{2}$ as a parent,
then the $\frac{m}{p}$-region limits to $\delta_R = 1$ where the value of $\frac{\phi}{2 \pi}$
is $\frac{m^-}{p^-}$, $\frac{m}{p}$, and $\frac{m^+}{p^+}$.
It remains to prove this and use the formula
\eqref{eq:formula} to determine the asymptotic behaviour of the regions
as they emanate from $\delta_R = 1$. 

Section \ref{sub:bifs} briefly showed that changes to the configuration of $g$
often translate to changes in the nature of the attractor.
For instance, several $\frac{m}{p}$-region boundaries in Fig.~\ref{fig:bifSetWithBoundaries}
appear to be boundaries of robust chaos.
Curves where one piece of $g$ has zero slope are significant,
for example at $\gamma_4$ the sausage-string structure appears cease. 
We speculate that this occurs because the corresponding
piece of $g$ changes from increasing to decreasing (compare \cite[Figure 8]{Si18e}).
It remains to develop a theory for the bifurcation structures
shown in Fig.~\ref{fig:bifSetWithBoundariesZoom}.
We expect that classical results on
one-dimensional expanding maps \cite{LaYo73,LiYo78}
could be used to formally identity parameter regions of robust chaos,
and establish ergodic properties of this chaos.

\section*{Acknowledgements}

This work was supported by Marsden Fund contract MAU2504 managed by Royal Society Te Ap\={a}rangi.
The author thanks Paul Glendinning for helpful conversations.

\appendix

\section{The first few regions $F_r$}
\label{app:nStarProof}

Here we prove Proposition \ref{pr:nStar}
by characterising the first few regions $F_r$, defined by \eqref{eq:Fr},
and relating them to the $x$-axis, $\Sigma_1$.

Recall, $s_r$ is the slope of $\Sigma_{-r} = f_R^{-r} \left( \Sigma_0 \right)$.
If $\rD f_R$ has complex eigenvalues (i.e.~$\delta_R > \frac{\tau^2}{4}$, so $X$ is a focus),
then we cannot have $s_r \le 0$ for all $r \ge 1$.
As in \cite{GlSi25}, let $q^*$ be the smallest $r \ge 1$
for which $s_r > 0$ or $s_r = \infty$.

Notice $s_1 = -\tau_R$, so $q^* = 1$ if and only if $\tau_R < 0$.
Our example points {\bf a} and {\bf b} have $q^* = 1$,
while point {\bf c} has $q^* = 2$, see Fig.~\ref{fig:divK}.

\begin{lemma}
If $\delta_R > \frac{\tau^2}{4}$,
then $s_1 < s_2 < \cdots < s_{q^*-1}$.
\label{le:srInc}
\end{lemma}

\begin{proof}
If $s_r < 0$, then $s_{r+1} > s_r$ by \eqref{eq:slopeRecurrenceRelation}.
Thus $\{ s_r \}_{r \ge 1}$ is increasing while taking values in $(-\infty,0]$.
\end{proof}

For all $r \ge 0$, the line $\Sigma_{-r}$ contains the points $Q_1^r$ and $Q_1^{r+1}$,
see \eqref{eq:intersectionFormula}.
Let $M_r$ be the ray obtained by cutting $\Sigma_{-r}$ at $Q_1^r$
and retaining the part that contains $Q_1^{r+1}$.
Let $M_r^{\rm seg}$ be the line segment from $Q_1^r$ to $Q_1^{r+1}$,
and let $M_r^{\rm rest} = M_r \setminus M_r^{\rm seg}$ be the rest of $M_r$, see Fig.~\ref{fig:divM}.
Also let $\tilde{\Sigma}_0$ be the part of $\Sigma_0$ with $y > -1$.

\begin{figure}[b!]
\begin{center}
\includegraphics[width=7.5cm]{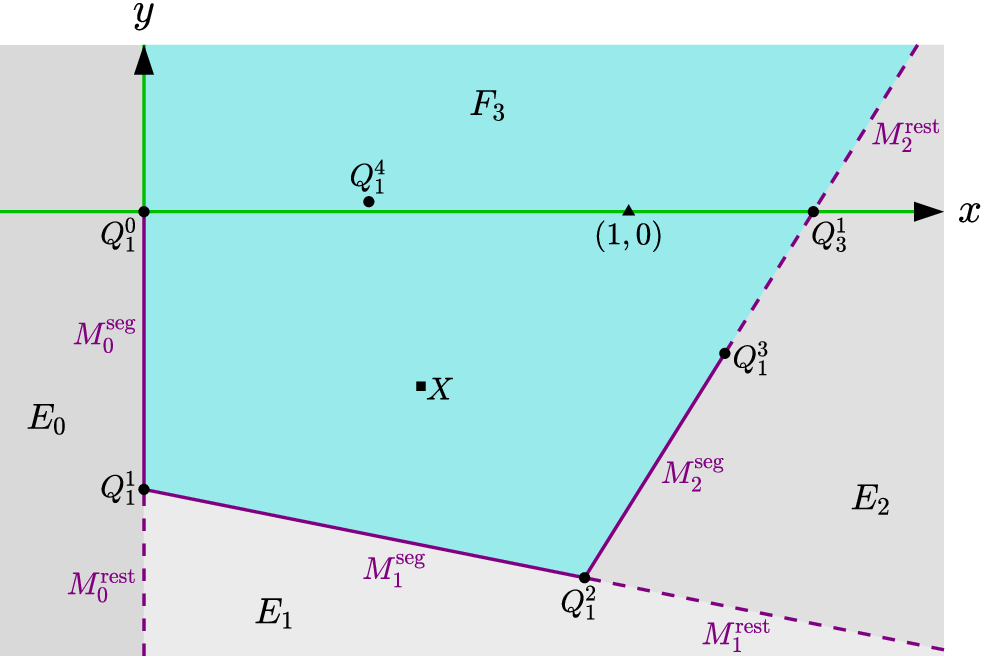}
\caption{
The phase space of \eqref{eq:f}
showing $F_3$ and other sets defined
in the text for parameter point {\bf c} of Fig.~\ref{fig:bifSet}. 
\label{fig:divM}
} 
\end{center}
\end{figure}

The following lemma characterises $F_2, F_3, \ldots, F_{q^*+1}$.
Note that $F_0 = \mathbb{R}^2$, and $F_1 = \Omega_R$.

\begin{lemma}
Suppose $\delta_R > 0$ and $\delta_R > \frac{\tau_R^2}{4}$.
For each $r = 1,2,\ldots,q^*+1$,
\begin{enumerate}[label=\roman*),ref=\roman*,itemsep=0mm]
\item 
$Q_1^r \in F_{r-1}$,
\item
if $r < q^*+1$, then $Q_1^r$ lies below $\Sigma_1$,
\item
if $r > 1$, then $F_r$ is the open set bounded in anticlockwise order by
$\tilde{\Sigma}_0$, $M_1^{\rm seg}, M_2^{\rm seg}, \ldots, M_{r-2}^{\rm seg}$, and $M_{r-1}$.
\end{enumerate}
Moreover, $E_{q^*} \cap \Sigma_1 = \left[ z_{q^*+1}, \infty \right) \times \{ 0 \}$.
\label{le:firstFewFr}
\end{lemma}

\begin{proof}
We prove the result by induction on $r$.
With $r = 1$, part (i) of Lemma \ref{le:firstFewFr} holds because $F_0 = \mathbb{R}^2$,
part (ii) holds because $Q_1^1 = (0,-1)$,
and (iii) holds vacuously.

Suppose the result is true for some $r = 1,2,\ldots,q^*$.
It remains to verify parts (i), (ii), and (iii) for $r+1$.

By the induction hypothesis, $Q_1^r$ belongs to $F_{r-1}$ and lies below $\Sigma_1$.
Thus, by \eqref{eq:Fr4}, $f_R^{-1} \left( Q_1^r \right) = Q_1^{r+1} \in F_r$, verifying (i) for $r+1$.

The region $F_{r+1}$ can be formed by
starting at the point $Q_1^r$, which belongs to $M_{r-1}$ on the boundary of $F_r$,
and cutting $F_r$ along $\Sigma_{-r}$.
The cut extends from $Q_1^r$ to infinity
with $x$-values tending to $\infty$, or, in the special case $s_r = \infty$,
with $x > 0$ constant and $y$-values tending to infinity.
In the case $r = 1$ this observation is elementary,
while in the case $r > 1$ it follows from part (iii) of the induction hypothesis
and the bound $s_r > s_{r-1}$ of Lemma \ref{le:srInc}.
Since $Q_1^{r+1}$ belongs to both $F_r$ and the $M_r$, the cut is along $M_r$.

Thus cut divides $F_r$ into a region bounded by $M_{r-1}^{\rm rest}$ and $M_r$, and a region bounded by
$\tilde{\Sigma}_0, M_1^{\rm seg}, M_2^{\rm seg}, \ldots, M_{r-1}^{\rm seg}$, and $M_r$.
Since $f_R^{r+1}$ maps $M_{r-1}^{\rm rest}$ into $\Omega_L$,
the first region is $E_r$ and the second region is $F_{r+1}$,
verifying (iii) for $r+1$.

If $r < q^*+1$, then the slope $s_r$ of $M_r$ is finite and non-positive.
Thus $Q_1^{r+1} \in M_r$ lies level with or below $Q_1^r$,
so lies below $\Sigma_1$, verifying (ii) for $r+1$.

In the case $r = q^*$,
the above arguments show that $E_{q^*}$ is the convex region
bounded by $M_{q^*-1}^{\rm rest}$ and $M_{q^*}$.
These rays emanate from $Q_1^{q^*}$ below $\Sigma_1$,
with $M_{q^*-1}^{\rm rest}$ horizontal or with negative slope,
and $M_{q^*}$ vertical or with positive slope.
Thus $M_{q^*}$ intersects $\Sigma_1$, necessarily at $Q_{q^*+1}^0 = \left( z_{q^*+1}, 0 \right)$,
and $E_{q^*} \cap \Sigma_1 = \left[ z_{q^*+1}, \infty \right) \times \{ 0 \}$.
\end{proof}

The last statement in Lemma \ref{le:firstFewFr}
shows that $I_{q^*+1} = [z_{q^*+1},0]$ if $z_{q^*+1} \le 1$,
and $I_{q^*+1} = \varnothing$ if $z_{q^*+1} > 1$.

\begin{lemma}
Suppose $\delta_R > 0$ and $\delta_R > \frac{\tau_R^2}{4}$.
Then $I_n = \varnothing$ for all $2 \le n \le q^*$ and $n^* > q^*$.
\label{le:nStarGreaterThanqStar}
\end{lemma}

\begin{proof}
By Lemma \ref{le:firstFewFr}(iii),
all parts of the boundary of $F_{q^*}$ lie below $\Sigma_1$ except $\tilde{\Sigma}_0$.
Thus $F_{q^*}$ contains the positive $x$-axis.
Thus, by \eqref{eq:Fr}, for any $2 \le n \le q^*$
the region $E_{n-1}$ has no points on the positive $x$-axis,
and so $I_n = \varnothing$ by \eqref{eq:In2}.
This implies $\chi(1,0) > q^* + 1$,
hence $n^* > q^*$ by \eqref{eq:nStar} and \eqref{eq:chiToN}.
\end{proof}

By Lemma \ref{le:nStarGreaterThanqStar},
either $n^* = q^* + 1$, such as for example points {\bf a} and {\bf b},
or $n^* > q^* + 1$, such as for example point {\bf c}.
The following result handles the latter case.
The case $n^* = q^* + 1$ is easier and handled within the proof of
Propositon \ref{pr:nStar} (given below).

\begin{lemma}
Suppose $\delta_R > 0$, $\delta_R > \frac{\tau_R^2}{4}$, and $n^* > q^*+1$.
For all $q^* + 1 \le r \le n^*$,
\begin{enumerate}[label=\roman*),ref=\roman*,itemsep=0mm]
\item 
$Q_1^r \in F_{r-1}$,
\item
$Q_1^{r-1}$ lies below $\Sigma_1$,
\item
if $r > q^* + 1$ then $F_r$ is the polygon with vertices
$Q_1^1, Q_1^2, \ldots, Q_1^{r-1}, Q_{r-1}^1$ (ordered anticlockwise), and
\item
if $r > q^* + 1$ then $E_{r-1} \cap \Sigma_1 = \left[ z_r, z_{r-1} \right) \times \{ 0 \}$.
\end{enumerate}
\label{le:nextFewFr}
\end{lemma}

\begin{proof}
We prove the result by induction on $r$.
Parts (i) and (ii) of Lemma \ref{le:nextFewFr} are true for $r = q^* + 1$ by Lemma \ref{le:firstFewFr},
while parts (iii) and (iv) are true for $r = q^* + 1$ vacuously.

Let $q^* + 1 \le s < n^*$,
and suppose Lemma \ref{le:nextFewFr} is true for all $q^*+1 \le r \le s$.
This is our induction hypothesis.
For brevity, we write ${\rm IH(i)}$
for the assumption that part (i) of Lemma \ref{le:nextFewFr} is true for all $q^*+1 \le r \le s$,
and write ${\rm IH(ii)}_r,\ldots,{\rm IH(iv)}_r$ analogously.
To complete the proof we verify (i)--(iv) for $r = s+1$.

By \eqref{eq:nStar} and \eqref{eq:chiToN}, $(1,0) \in E_{n^* - 1}$.
Thus $(1,0) \notin E_r$ for any $r \le n^* - 2$.
Thus $z_s > 1$ by IH(iv) and the last statement in Lemma \ref{le:firstFewFr}.

We now consider the part of the boundary of $F_s$ formed by $\Sigma_{-(s-1)}$, call it $S$.
In the case $s = q^* + 1$,
$S$ is the ray $M_{q^*}$ by Lemma \ref{le:firstFewFr}(iii).
This ray extends upwards from $Q_1^{q^*}$,
which is situated below $\Sigma_1$ by Lemma \ref{le:firstFewFr}(ii).
In the case $s > q^* + 1$,
$S$ is the line segment from $Q_1^{s-1}$ to $Q_{s-1}^1$ by IH(iii).
By IH(ii), $Q_1^{s-1}$ lies below $\Sigma_1$,
while $Q_{s-1}^1 = \left( 0, z_{s-1} - 1 \right)$ lies above $\Sigma_1$
because $z_{s-1} > z_s > 1$ by IH(iv).
In either case $S$ intersects $\Sigma_1$
at $Q_s^0 = (z_s,0)$ by \eqref{eq:znAsIntersection}.

Notice $S = M_{s-1} \cap \overline{\Omega}_R$, 
and recall $Q_1^s \in M_{s-1}$.
Since $Q_1^s \in F_{s-1}$, by IH(i) and Lemma \ref{le:firstFewFr}(i),
we have $Q_1^s \in \Omega_R$, and so $Q_1^s \in S$.
Also $Q_s^1 = (0, z_s - 1)$ lies above $\Sigma_1$ because $z_s > 1$.

Now suppose for a contradiction that $Q_1^s$ lies on or above $\Sigma_1$.
Let $H$ be the interior of the convex hull of
$Q_1^1, Q_1^2, \ldots, Q_1^s$ and $Q_s^1$.
Since $f_R$ is affine, $f_R(H)$ is the polygon with vertices
$Q_1^0, Q_1^1, \ldots, Q_1^{s-1}$ and $Q_s^0$.
This polygon is the part of $H$ that lies below $\Sigma_1$, thus $f_R(H) \subset H$,
which is a contradiction because $\delta_R > 1$ so $f_R$ is area-expanding.
Thus $Q_1^s$ lies below $\Sigma_1$, verifying (ii) for $r = s+1$.
By IH(i), $Q_1^s \in F_{s-1}$.
Thus by \eqref{eq:Fr4}, $f_R^{-1} \left( Q_1^s \right) = Q_1^{s+1} \in F_s$,
verifying (i) for $r = s+1$.

The line $\Sigma_{-s}$ cuts $F_s$ into two pieces: $E_s$ and $F_{s+1}$.
This line passes through $Q_1^s$ and $Q_s^1$,
where, as we have just shown,
$Q_1^s$ lies below $\Sigma_1$ and belongs to $S$,
and $Q_s^1$ lies above $\Sigma_1$ and belongs to the part of the boundary of $F_s$ formed by $\Sigma_0$.
Thus, by IH(iii) and Lemma \ref{le:firstFewFr}(iii),
$F_{s+1}$ is the polygon with
vertices $Q_1^1, Q_1^2, \ldots, Q_1^s, Q_s^1$, verifying (iii) for $r = s+1$.

This construction also shows that the edge from $Q_1^s$ and $Q_s^1$ intersects $\Sigma_1$,
necessarily at $Q_{s+1}^0 = (z_{s+1},0)$ by \eqref{eq:znAsIntersection}.
So $F_{s+1} \cap \Sigma_1 = \left( 0, z_{s+1} \right) \times \{ 0 \}$,
hence $E_s \cap \Sigma_1 = \left[ z_{s+1}, z_s \right) \times \{ 0 \}$,
verifying (iv) for $r = s+1$.
\end{proof}

\begin{proof}[Proof of Proposition \ref{pr:nStar}]
In the case $n^* = q^* + 1$,
$E_{n^* - 1} \cap \Sigma_1 = \left[ z_{n^*}, \infty \right) \times \{ 0 \}$
by the last part of Lemma \ref{le:firstFewFr},
while in the case $n^* > q^* + 1$,
$E_{n^*-1} \cap \Sigma_1 = \left[ z_{n^*}, z_{n^*-1} \right)$ by Lemma \ref{le:nextFewFr}(iv),
where $z_{n^*-1} > 1$.
In either case, $E_{n^*-1} \cap \tilde{\Sigma}_1 = \left[ z_{n^*}, 1 \right]$ by \eqref{eq:In2},
and where $z_{n^*} \le 1$ by the definition of $n^*$.
This verifies part (ii) of Proposition \ref{pr:nStar}.
Furthermore, $I_n = \varnothing$ for all $2 \le n \le n^*-1$
by Lemma \ref{le:nStarGreaterThanqStar}, Lemma \ref{le:nextFewFr}(iv),
and the last part of Lemma \ref{le:firstFewFr},
verifying part (i) of Proposition \ref{pr:nStar}.

Notice \eqref{eq:Fr4} implies $F_{r+1} = f_R^{-1} \left( F_r \cap U_R^{-1} \right)$,
where $U_R^{-1}$ consists of all points below $\Sigma_1$.
From the description of $F_{n^*}$ in Lemma \ref{le:firstFewFr}(iii) in the case $n^* = q^* + 1$,
and in Lemma \ref{le:nextFewFr}(iii) in the case $n^* > q^* + 1$,
$F_{n^*} \cap U_R^{-1}$ is the polygon with vertices 
$Q_1^0, Q_1^1, \ldots, Q_1^{n^*-1}, Q_{n^*-1}^1$.
Thus $F_{n^*+1}$ 
is the polygon with vertices $Q_1^1,Q_1^2,\ldots,Q_1^{n^*},Q_{n^*}^1$,
and notice $Q_{n^*}^1$ lies on or below $\Sigma_1$ because $z_{n^*} \le 1$.
\end{proof}

{\footnotesize
\bibliographystyle{unsrt}
\bibliography{ZeroEigCharIndMapBIB}
}

\end{document}